\documentclass[11pt]{article}

\newcommand{\hide}[1]{}

\usepackage[style=numeric]{biblatex}
\usepackage{fancybox,graphicx}
\usepackage[title]{appendix}
\usepackage{xcolor}
\usepackage{url}
\usepackage[linkbordercolor={1 0 0}]{hyperref}
\usepackage{mathtools}
\usepackage{amssymb}
\usepackage{algorithm}
\usepackage[noend]{algpseudocode}
\usepackage{authblk}

\hide{
\usepackage{amscd}
\usepackage{amsfonts}
\usepackage{amsmath}
\usepackage{amssymb}
\usepackage{amsthm}
\usepackage{cases}		
\usepackage{cutwin}
\usepackage{enumerate}
\usepackage{epstopdf}
\usepackage{graphicx}
\usepackage{ifthen}
\usepackage{lipsum}
\usepackage{mathrsfs}	
\usepackage{multimedia}
\usepackage{wrapfig}
}
\usepackage{amsthm}
\newtheorem{theorem}{Theorem}[section]

\newtheorem{lemma}[theorem]{Lemma}

\newlength\tindent
\title { A Unifying Framework of Accelerated First-Order Approach to Strongly Monotone Variational Inequalities }

\author{
Kevin Huang\thanks{Department of Industrial and System Engineering, University of Minnesota, huan1741@umn.edu}
\hspace{1cm}
Shuzhong Zhang\thanks{Department of Industrial and System Engineering, University of Minnesota, zhangs@umn.edu}
}

\date{\today}

\begin{document}

\maketitle

\begin{abstract}
    In this paper, we propose a unifying framework incorporating several momentum-related search directions for solving strongly monotone variational inequalities. The specific combinations of the search directions in the framework are made to guarantee the optimal iteration complexity bound of $\mathcal{O}\left(\kappa\ln(1/\epsilon)\right)$ to reach an $\epsilon$-solution, where $\kappa$ is the condition number. This framework provides the flexibility for algorithm designers to train -- among different parameter combinations -- the one that best suits the structure of the problem class at hand. The proposed framework includes the following iterative points and directions as its constituents: the extra-gradient, the optimistic gradient descent ascent (OGDA) direction (aka ``optimism''), the ``heavy-ball'' direction, and Nesterov's extrapolation points. As a result, all the afore-mentioned methods become the special cases under the general scheme of extra points. We also specialize this approach to strongly convex minimization, and show that a similar extra-point approach achieves the optimal iteration complexity bound of $\mathcal{O}(\sqrt{\kappa}\ln(1/\epsilon))$ for this class of problems.

    \vspace{3mm}

\end{abstract}

\section{Introduction}
\label{sec:intro}
In this paper, we are concerned with solving the following variational inequality (VI) problem: Given a constraint set $\mathcal{Z}\subseteq \mathbb{R}^n$ and a mapping $F:\mathcal{Z}\rightarrow \mathbb{R}^n$, find $z^*\in\mathcal{Z}$ such that
\begin{equation}
    \label{vi-prob}
    F(z^*)^{\top}(z-z^*)\geq 0,\quad\forall z\in\mathcal{Z}.
\end{equation}
The study of VI problem \eqref{vi-prob} dates back to 1960's, first in the form of complementarity problem (CP) which models various equilibrium settings such as the economic supply-demand equilibria, traffic equilibria, and generally the Nash equilibria. For an extensive introduction to VI and its applications, we refer the readers to~\cite{facchinei2007finite} and the references therein. It may be helpful to recite some of the famous VI modeling applications:

\begin{description}
\item[Complementarity Problem (CP)]

Finding $z^* \ge 0$ such that $F(z^*) \ge 0$ and $F_i(z^*)z^*_i=0$ $\forall i$ is equivalent to finding $z^* \ge 0$ such that
\[
 F(z^*)^\top (z - z^*) \ge 0, \,\,\, \forall z \ge 0.
\]

\item[Equation Solving] Finding $F(z^*)=0$ is equivalent to finding $z^*$ such that
\begin{equation*}
F(z^*)^\top (z - z^*) \ge 0, \,\,\, \forall z .
\end{equation*}

\item[Constrained Optimization] Finding a first-order optimal solution for
 $   \min\limits_{x\in\mathcal{X}}f(x)$,
where $f(x)$ is differentiable and $\mathcal{X}$ is a convex set, is equivalent to finding $x^* \in \mathcal{X}$ such that
\begin{equation*}
    \nabla f(x^*)^{\top}(x-x^*)\geq 0,\quad \forall x \in\mathcal{X}.
\end{equation*}

\item[The Minimax Saddle-Point Problem] The first-order condition for
\begin{equation*}
    \min\limits_{x\in\mathcal{X}}\max\limits_{y\in\mathcal{Y}}f(x,y),
\end{equation*}
where $f(x,y)$ is differentiable in both $x,y$ and $\mathcal{X},\mathcal{Y}$ are both convex, is to find $(x^*,y^*) \in \mathcal{X} \times \mathcal{Y}$ such that
\begin{equation*}
    \begin{pmatrix}
    \nabla_x f(x^*,y^*)\\
    -\nabla_y f(x^*,y^*)
    \end{pmatrix}^{\top}
    \begin{pmatrix}
    x-x^*\\
    y-y^*
    \end{pmatrix}\geq 0,\quad\forall (x,y) \in\mathcal{X} \times \mathcal{Y}.
\end{equation*}

\end{description}

In this paper, we consider the VI model~\eqref{vi-prob} where $\mathcal{Z}$ is a closed convex set. Moreover, throughout this paper the following two conditions are assumed:
\begin{equation}
    \label{pre:strong-mono}
    \left(F(z)-F(z')\right)^{\top}(z-z')\geq \mu\|z-z'\|^2, \quad\forall z,z'\in\mathcal{Z},
\end{equation}
for some $\mu>0$, and
\begin{equation}
    \label{pre:lip}
    \|F(z)-F(z')\|\leq L\|z-z'\|,\quad\forall z,z'\in \mathcal{Z},
\end{equation}
for some $L \ge \mu>0$. Condition~\eqref{pre:strong-mono} is known as the {\it strong monotonicity}\/ of $F$, while Condition~\eqref{pre:lip} is known as the {\it Lipschitz continuity}\/ of $F$.

Let us denote
    $\kappa := \frac{L}{\mu}\ge 1$ and  $\sigma :=\frac{\mu}{L}=\frac{1}{\kappa}\le 1$.
Parameter $\kappa$ is also usually known as the condition number of \eqref{vi-prob}.

There exists a considerable amount of literature on the theory and algorithms for solving finite-dimensional VI models. Many of the results before 2007 can be found in the celebrated monograph of Facchinei and Pang~\cite{facchinei2007finite}. Recent years have seen a renewed interest on the topic, due to VI's connections to the first-order methods for optimization in the context of machine learning and statistics. The current paper aims to propose a new framework of the first-order algorithms for finite-dimensional VI in that context. Our study revolves around the issue of {\it iteration complexity}\/ for solving model~\eqref{vi-prob}. To set the stage for our discussion, let us first note that for the strongly monotone VI model~\eqref{vi-prob} a unique solution $z^*$ exists. We call a solution $z$ to be an $\epsilon$-solution if $\|z-z^*\| \le \epsilon$, where $\epsilon>0$ is a given precision. The so-called iteration complexity analysis for an algorithm is to upper bound the number of iterations required to find an $\epsilon$-solution for any VI instances within the class of strongly monotone VI. Before proceeding, we shall note a recent {\it lower bound}\/ result. Zhang {\it et al.}~\cite{junyu2019} considers
strongly-convex-strongly-concave saddle point problems, and shows that there exists a class of problems in such a way that no first-order algorithm can find an $\epsilon$-solution in less than \begin{equation*}
    \Omega\left(\sqrt{\frac{L_x}{\mu_x}+\frac{L^2_{xy}}{\mu_x\mu_y}+\frac{L_y}{\mu_y}}\ln\left(\frac{1}{\epsilon}\right)\right)
\end{equation*}
iterations, where
$L_x$ and $L_y$ are the Lipschitz constants for $\nabla_xf(\cdot,y)$ and $\nabla_yf(x,\cdot)$ respectively for fixed $y/x$,
and $L_{xy}$ is the Lipschitz constant for $\nabla_xf(x,\cdot)/\nabla_yf(\cdot,y)$ for fixed $x/y$.  As we noted earlier, the above is a special case of the strongly monotone VI model.
In the context of VI, if $L_x=L_y=L_{xy}=L$ and $\mu_x=\mu_y=\mu$, then the above lower bound can be regarded as a lower bound result for general strongly monotone VI problems as well. Therefore, a lower bound for the strongly monotone VI is
    $\Omega\left(\kappa\ln\left(\frac{1}{\epsilon}\right)\right)$.
This is in sharp contrast to the case of strongly convex optimization, where the lower bound for the iteration complexity is $\Omega\left(\sqrt{\kappa} \ln\left(\frac{1}{\epsilon}\right)\right)$ and Nesterov's accelerated gradient method has achieved this iteration bound, hence known as an ``optimal'' algorithm (see~\cite{nesterov2003introductory}).

On the side of algorithm developments for VI, classical methods include the projection algorithm (see~\cite{facchinei2007finite}), the proximal point method as proposed by Martinet \cite{martinet1970breve} and popularized by Rockafeller \cite{rockafellar1976monotone}, and the matrix splitting method of Tseng \cite{tseng1995linear}. In this paper we shall focus on the methods involving projection of search directions onto the feasible set. In fact, there have been quite some interesting variants and configurations of the vanilla gradient projection in the literature. For example, the so-called {\it extra-gradient method}\/ was proposed by Korpelevich in 1976~\cite{korpelevich1976extragradient} for the saddle point problems
which was shown to be linearly convergent for strongly monotone VI by Tseng in~\cite{tseng1995linear}.
Another interesting method, known as the optimistic gradient descent ascent (OGDA) method proposed by Popov in 1980~\cite{popov1980modification}. 
That method has a close relation with the so-called momentum-related methods, and we shall come back to this point later. The convergence of OGDA for strongly-convex-strongly-concave saddle point problem was studied by Mokhtari {\it et al.}\/ in \cite{mokhtari2019unified} (the convergence for the convex-concave case was studied by the same authors in \cite{mokhtari2020convergence}) and a related proof for strongly monotone VI can be found in Palaniappan and Bach~\cite{palaniappan2016stochastic}. The above mentioned extra-gradient method and OGDA achieve an iteration complexity of $\mathcal{O}(\kappa\ln(1/\epsilon))$ for an $\epsilon$-solution. In this sense, the extra-gradient method and the OGDA method are optimal for solving the strongly monotone VI problems.
In this paper, we refer to these methods which attain the lower bound complexity as {\it accelerated first-order}\/ methods, since they gain an accelerated rate comparing to the vanilla projection method, which has an $\mathcal{O}(\kappa^2\ln(1/\epsilon))$ complexity bound. More discussions on this issue can be found in Section \ref{sec:proj-and-extra-grad}.

Speaking of acceleration, Nesterov proposed the first method of this type in 1983 for convex optimization~\cite{nesterov1983method}.
There has been an intensive recent research effort on the subject; see the recent monograph \cite{d2021acceleration} for a comprehensive survey. We devote Section~\ref{sec:opt} of our paper to the technical background of the subject in the context of optimization.
At this point, we shall introduce two particular methods which we shall borrow to solve the VI models.
The first one is Nesterov's gradient acceleration
method itself~\cite{nesterov1983method}, which will also be referred to as ``Nesterov's accelerated method'' or ``Nesterov's method'' in the context of VI in this paper. The method was initially proposed for non-strongly (but smooth) convex optimization, yielding an $\mathcal{O}(1/\sqrt{\epsilon})$ iteration complexity bound, compared to normal gradient descent method with the bound $\mathcal{O}(1/\epsilon)$. For solving strongly convex optimization, the method can be modified to yield an $\mathcal{O}(\sqrt{\kappa}\ln(1/\epsilon))$ iteration bound~\cite{nesterov2003introductory}. These bounds are shown to be optimal in the context of convex optimization. The second one predates Nesterov's acceleration, and leverages on the so-called momentum of the dynamics in the gradient fields. The method was introduced by Polyak in 1964~\cite{polyak1964some} and is more commonly known by the name of ``heavy-ball'' method, which can be shown to yield an iteration complexity $\mathcal{O}(\sqrt{\kappa}\ln(1/\epsilon))$ for minimizing a strongly convex quadratic objective function.

This paper aims for a unified scheme of {\it extra points} to solve strongly monotone VI models, in that all the insightful search directions based on the ideas of momentum (``heavy-ball'' and ``optimism'') and extra steps (Nesterov's extrapolation point and the extra-gradient step) are allowed to collaborate to yield an even better and steadier performance. For ease of referencing, we shall call the proposed approach {\it the extra-point method}. The main motivation underlying the scheme is to have those individual search steps tuned from one problem class to the other, and we believe that it makes sense to learn from the data and to find out a suitable configuration of those first-order directions for each given problem class needing to be tackled. It is therefore meaningful to find an inclusive and general enough framework under which the {\it optimal}\/ iteration bound is still achievable. The main theme of this paper is to study strongly monotone VI. However, in Section~\ref{sec:opt} we also analyze an important special case: convex optimization, and we shall extend the general framework to this important subclass of problems. Our proposed approach relies on a dynamic process to manage the extra sequences of iterative points to engage all the search directions mentioned above.

The rest of the paper is organized as follows. In Section \ref{sec:his} we shall illustrate what motivates the ideas behind the extra-point approach, by analyzing the classical projection method, the extra-gradient method, and other accelerated methods. In Section \ref{sec:vi} we present the extra-point approach for strongly monotone VI with an iteration complexity analysis.
In Section \ref{sec:opt}, we present an extra-point scheme to solve strongly convex optimization with enhanced iteration complexity accordingly.

\section{An Analysis of the Projection Method}
\label{sec:his}

In this section we shall conduct an analysis revealing the mechanism leading to various phenomena of acceleration for strongly monotone VI. The analysis motivates the developments to be presented in Section \ref{sec:vi}.
First, we introduce the projection operator $P_{\mathcal{Z}}(\cdot)$ and its properties:
\begin{equation}
    \label{pre:proj}
    P_{\mathcal{Z}}(z) = \arg\min\limits_{z'\in\mathcal{Z}}\|z-z'\|^2.
\end{equation}
We note the following well-known properties (the non-expansiveness and the 1-co-coerciveness) for the projection operator (cf.~e.g.\ Proposition 4.4 in  \cite{bauschke2011convex}).
For a convex set $\mathcal{Z}$ and the projection operator $P_{\mathcal{Z}}(\cdot)$ defined as \eqref{pre:proj}, we have:
\begin{eqnarray}
    \label{pre:non-exp}
 \mbox{ Non-expansiveness: } & &   \|P_{\mathcal{Z}}(z)-P_{\mathcal{Z}}(z')\|\leq \|z-z'\|, \quad\forall z,z' \\
 \mbox{ The 1-co-coerciveness: } & &
    \label{1-co-coer}
    \|P_{\mathcal{Z}}(z)-P_{\mathcal{Z}}(z')\|^2\leq \left(P_{\mathcal{Z}}(z)-P_{\mathcal{Z}}(z')\right)^{\top}(z-z'), \quad\forall z,z'.
\end{eqnarray}

\subsection{Vanilla gradient projection and the extra-gradient method}
\label{sec:proj-and-extra-grad}

Let us first take a close look at the vanilla projection method:
\[
    z^{k+1}=P_{\mathcal{Z}}\left(z^k-\alpha F(z^k)\right),\quad k=0,1,2,....
\]
A standard analysis gives the following estimation: 
\begin{eqnarray*}
         \|z^{k+1}-z^*\|^2 & = &\|P_{\mathcal{Z}}\left(z^k-\alpha F(z^k)\right)- P_{\mathcal{Z}}\left(z^*-\alpha F(z^*) \right)\|^2  \\
         & \leq & \|z^k-\alpha F(z^k)-z^*+\alpha F(z^*)\|^2 \\
         & = & \|z^k-z^*\|^2-2\alpha(z^k-z^*)^{\top}\left(F(z^k)-F(z^*)\right)+\alpha^2\|F(z^k)-F(z^*)\|^2 \\
         & \leq & (1-2\alpha\mu+\alpha^2L^2)\|z^k-z^*\|^2 \\
         & \overset{\alpha:=\frac{\mu}{L^2}}{=} & (1-\sigma^2)\|z^k-z^*\|^2.
\end{eqnarray*}
The first inequality uses the non-expansiveness of the projection \eqref{pre:non-exp}, and the second inequality uses the strong monotonicity \eqref{pre:strong-mono} and Lipschitz continuity \eqref{pre:lip}. This results in a linear convergence rate of $1-\sigma^2$.

The term
$-2\alpha(z^k-z^*)^{\top}\left(F(z^k)-F(z^*)\right)$ would guarantee a linear rate $1-\alpha\mu$
if there were no $\alpha^2\|F(z^k)-F(z^*)\|^2$ term. The presence of the last term
causes a smaller step size of $\alpha=\frac{\mu}{L^2}$, leading to a reduction rate of $1-\sigma^2$.

The idea behind the extra-gradient method is to introduce another point to evaluate the gradient so as to avoid dealing with the last term as such. The extra-gradient method proceeds as follows:
\begin{equation}
    \label{ex-grad-1} \left\{
    \begin{array}{cll}
         z^{k+0.5} &=& z^k-\alpha F(z^k), \\
         z^{k+1} &=& P_{\mathcal{Z}}\left(z^k-\alpha F(z^{k+0.5})\right),\quad k=0,1,2,...
    \end{array} \right.
\end{equation}

Similarly, the analysis then goes:
\begin{eqnarray}
    \label{ex-grad-1-step-1}
         \|z^{k+1}-z^*\|^2 & \leq & \|P_{\mathcal{Z}}\left(z^k-\alpha F(z^{k+0.5})\right)-P_{\mathcal{Z}}(z^*)\|^2  \nonumber \\
         & \leq & \|z^k-\alpha F(z^{k+0.5})-z^*\|^2 \nonumber \\
         & = & \|z^k-z^*\|^2-2\alpha(z^k-z^*)^{\top} F(z^{k+0.5})+\alpha^2\|F(z^{k+0.5})\|^2 \nonumber \\
         & = & \|z^k-z^*\|^2-2\alpha(z^{k+0.5}-z^*)^{\top}F(z^{k+0.5})+\alpha^2\|F(z^{k+0.5})\|^2 \nonumber \\
         && -2\alpha(z^k-z^{k+0.5})^{\top}F(z^{k+0.5}).
\end{eqnarray}
Note that
\begin{eqnarray*}
         -2\alpha(z^{k+0.5}-z^*)^{\top}F(z^{k+0.5}) & \leq & -2\alpha(z^{k+0.5}-z^*)^{\top}F(z^{*})-2\alpha\mu\|z^{k+0.5}-z^*\|^2 \\
         & \leq & -2\alpha\mu\|z^{k+0.5}-z^*\|^2 \\
         & \leq & -\alpha\mu\|z^k-z^*\|^2+2\alpha\mu\|z^k-z^{k+0.5}\|^2
\end{eqnarray*}
and
\begin{eqnarray*}
         \alpha^2\|F(z^{k+0.5})\|^2 & = & \alpha^2\|F(z^{k+0.5})-F(z^k)+F(z^k)\|^2  \\
         & = & \alpha^2\|F(z^{k+0.5})-F(z^k)+\frac{1}{\alpha}(z^k-z^{k+0.5})\|^2 \\
         & = & \alpha^2\|F(z^{k+0.5})-F(z^k)\|^2+2\alpha\left(F(z^{k+0.5})-F(z^k)\right)^{\top}(z^k-z^{k+0.5})+\|z^k-z^{k+0.5}\|^2 \\
         & \leq & 2\alpha\left(F(z^{k+0.5})-F(z^k)\right)^{\top}(z^k-z^{k+0.5})+(\alpha^2L^2+1)\|z^k-z^{k+0.5}\|^2
\end{eqnarray*}
and
\begin{eqnarray*}
         -2\alpha(z^k-z^{k+0.5})^{\top}F(z^{k+0.5}) & = & -2\alpha(z^k-z^{k+0.5})^{\top}\left(F(z^{k+0.5})-F(z^k)+F(z^k)\right) \\
         & = & -2\alpha(z^k-z^{k+0.5})^{\top}\left(F(z^{k+0.5})-F(z^k)+\frac{1}{\alpha}(z^k-z^{k+0.5})\right) \\
         & = & -2\alpha(z^k-z^{k+0.5})^{\top}\left(F(z^{k+0.5})-F(z^k)\right)-2\|z^k-z^{k+0.5}\|^2.
\end{eqnarray*}
By substituting the above three bounds into \eqref{ex-grad-1-step-1} we obtain
\begin{eqnarray*}
         \|z^{k+1}-z^*\|^2 & \leq & (1-\alpha\mu)\|z^k-z^*\|^2+(\alpha^2L^2+2\alpha\mu-1)\|z^{k}-z^{k+0.5}\|^2  \\
         & \overset{\alpha:=\frac{1}{4L}}{\leq} & \left(1-\frac{\sigma}{4}\right)\|z^{k}-z^*\|^2.
\end{eqnarray*}

Note that in \eqref{ex-grad-1}, we assumed the domain of $F(\cdot)$ is 
$\mathbb{R}^n$, since the point $z^{k+0.5}$ is not necessarily in $\mathcal{Z}$. The main sequence $\{z^k\}$ for $k=0,1,2,...$ is, however, still within $\mathcal{Z}$. In order to deal with the case where the mapping $F(\cdot)$ is \textit{only} defined on $\mathcal{Z}$, another projection operator should also be applied to the sequence $\{z^{k+0.5}\}$, leading to the following variant of the extra-gradient method:
\begin{equation}
    \label{ex-grad-2} \left\{
    \begin{array}{ccl}
         z^{k+0.5} &=& P_{\mathcal{Z}}\left(z^k-\alpha F(z^k)\right), \\
         z^{k+1}   &=& P_{\mathcal{Z}}\left(z^k-\alpha F(z^{k+0.5})\right),\quad k=0,1,2,...
    \end{array} \right.
\end{equation}

With a slight modification, the same convergence result can be shown to hold. At the first step, we shall use the 1-co-coerciveness \eqref{1-co-coer} instead of non-expansiveness of the projection:
\begin{eqnarray}
    \label{ex-grad-2-step-1}
         \|z^{k+1}-z^*\|^2 & = & \|P_{\mathcal{Z}}\left(z^k-\alpha F(z^{k+0.5})\right)-P_{\mathcal{Z}}(z^*)\|^2  \nonumber \\
         & \leq & (z^{k+1}-z^*)^{\top}\left(z^k-z^*-\alpha F(z^{k+0.5})\right) \nonumber \\
         & = & \frac{1}{2}\|z^{k+1}-z^*\|^2+\frac{1}{2}\|z^k-z^*\|^2-\frac{1}{2}\|z^{k+1}-z^k\|^2 \nonumber \\
         && - \alpha(z^{k+1}-z^{k+0.5})^{\top}F(z^{k+0.5})-\alpha(z^{k+0.5}-z^{*})^{\top}F(z^{k+0.5}) .
\end{eqnarray}
We bound the last term with strong monotonicity of $F(\cdot)$:
\begin{eqnarray*}
         -\alpha(z^{k+0.5}-z^*)^{\top}F(z^{k+0.5}) & \leq & -\alpha(z^{k+0.5}-z^*)^{\top}F(z^{*})-\alpha\mu\|z^{k+0.5}-z^*\|^2 \\
         & \leq & -\alpha\mu\|z^{k+0.5}-z^*\|^2 \\
         & \leq & -\frac{1}{2}\alpha\mu\|z^k-z^*\|^2+\alpha\mu\|z^k-z^{k+0.5}\|^2.
\end{eqnarray*}

To bound the term $- \alpha(z^{k+1}-z^{k+0.5})^{\top}F(z^{k+0.5})$, we use the optimality condition of the projection for the update $z^{k+0.5}$:
\begin{equation*}
    \langle z^{k+0.5}-z^k+\alpha F(z^k),z-z^{k+0.5}\rangle\geq 0,\quad\forall z\in\mathcal{Z}.
\end{equation*}
Taking $z=z^{k+1}$ in the above inequality and rearranging terms, we have
\begin{equation*}
    -\alpha F(z^k)^{\top}(z^{k+1}-z^{k+0.5})\leq \frac{1}{2}\|z^k-z^{k+1}\|^2-\frac{1}{2}\|z^{k+0.5}-z^k\|^2-\frac{1}{2}\|z^{k+1}-z^{k+0.5}\|^2.
\end{equation*}
Therefore,
\begin{eqnarray*}
& &        - \alpha(z^{k+1}-z^{k+0.5})^{\top}F(z^{k+0.5})  \\
&=&  - \alpha(z^{k+1}-z^{k+0.5})^{\top}F(z^{k+0.5}-F(z^k))- \alpha(z^{k+1}-z^{k+0.5})^{\top}F(z^{k})  \\
         & \leq & - \alpha(z^{k+1}-z^{k+0.5})^{\top}F(z^{k+0.5}-F(z^k))
          +\frac{1}{2}\|z^k-z^{k+1}\|^2-\frac{1}{2}\|z^{k+0.5}-z^k\|^2-\frac{1}{2}\|z^{k+1}-z^{k+0.5}\|^2 \\
         & \leq & \alpha L \|z^{k+1}-z^{k+0.5}\|\|z^{k+0.5}-z^k\|
          +\frac{1}{2}\|z^k-z^{k+1}\|^2-\frac{1}{2}\|z^{k+0.5}-z^k\|^2-\frac{1}{2}\|z^{k+1}-z^{k+0.5}\|^2 \\
         & \leq & \frac{1}{2}\alpha^2L^2\|z^{k+0.5}-z^k\|^2+\frac{1}{2}\|z^{k+1}-z^{k+0.5}\|^2
          +\frac{1}{2}\|z^k-z^{k+1}\|^2-\frac{1}{2}\|z^{k+0.5}-z^k\|^2-\frac{1}{2}\|z^{k+1}-z^{k+0.5}\|^2 \\
         & = & \frac{1}{2}\alpha^2L^2\|z^{k+0.5}-z^k\|^2+\frac{1}{2}\|z^k-z^{k+1}\|^2-\frac{1}{2}\|z^{k+0.5}-z^k\|^2.
\end{eqnarray*}

Combining the above bounds into \eqref{ex-grad-2-step-1} and multiplying both sides by 2 and rearranging, we obtain
\begin{eqnarray*}
         \|z^{k+1}-z^*\|^2 & \leq & (1-\alpha\mu)\|z^k-z^*\|^2+(\alpha^2L^2+2\alpha\mu-1)\|z^{k}-z^{k+0.5}\|^2  \\
         & \overset{\alpha:=\frac{1}{4L}}{\leq} & \left(1-\frac{\sigma}{4}\right)\|z^{k}-z^*\|^2,
\end{eqnarray*}
which is exactly the same reduction rate as for the regular extra-gradient method \eqref{ex-grad-1}.

\subsection{Other accelerated gradient methods}

There are a few other first-order algorithms that belong to the category of ``accelerated gradient methods'', broadly defined. Those methods were originally proposed not for solving VI problems. However, they can be easily adopted as such.

\subsubsection{The optimistic gradient descent ascent (OGDA) method}

Unlike extra-gradient method, which updates iterate $k$ with the mapping at $z^{k+0.5}$, the OGDA method updates with an extrapolated mapping direction:
\begin{equation}
    \label{update-ogda}
    z^{k+1} = P_{\mathcal{Z}}\left(z^k-\alpha F(z^k)-\tau\left(F(z^k)-F(z^{k-1})\right)\right).
\end{equation}
The following theorem shows the convergence result of the OGDA method, with the proof relegated to Appendix \ref{app:ogda}. A similar proof can be found in the appendix of \cite{palaniappan2016stochastic}.
\begin{theorem}
\label{th:ogda}
For solving a VI problem defined in \eqref{vi-prob}, with the mapping $F:\mathcal{Z}\rightarrow\mathbb{R}^n$ being Lipschitz continuous with constant $L$ and strongly monotone with constant $\mu$, the sequence $\{z^k\}$, $k=0,1,2,...$ generated from OGDA method \eqref{update-ogda} with parameters
    $\alpha = \frac{1}{2L}$ and $\tau = \frac{\alpha}{1+\sigma}$
yields an R-linear convergence as follows
\begin{equation*}
    \|z^k-z^*\|^2\leq 2(1+\sigma)^{-k}\|z^0-z^*\|^2.
\end{equation*}
\end{theorem}
Note that the term $F(z^k)-F(z^{k-1})$ is known as the \textit{optimism}.
The next two methods were initially introduced for optimization. To make a distinction in notation, below we use $x$ as the decision variable for optimization models.

\subsubsection{The heavy-ball method}

The heavy-ball method proposed in \cite{polyak1964some} was designed to solve strongly convex optimization
$\min\limits_{x} f(x)$ with the update rule:
\begin{equation}
    \label{heavy-ball-update}
    x^{k+1} = x^k - \frac{4}{(\sqrt{L}+\sqrt{\mu})^2}\nabla f(x^k)+\left(\frac{\sqrt{L}-\sqrt{\mu}}{\sqrt{L}+\sqrt{\mu}}\right)(x^k-x^{k-1}).
\end{equation}

The heavy-ball method is known to have an improved rate of convergence as compared to the regular gradient method (see~\cite{polyak1964some}), if $f$ is a strongly convex quadratic function. In particular, we have
\begin{equation}
    \label{heavy-ball-converge}
    \left\|\begin{pmatrix}
    x^{k+1}-x^*\\
    x^k-x^*
    \end{pmatrix}\right\|\leq C\cdot\left(\frac{1-\sqrt{\sigma}}{1+\sqrt{\sigma}}\right)^k\cdot \left\|\begin{pmatrix}
    x^{k}-x^*\\
    x^{k-1}-x^*
    \end{pmatrix}\right\|,
\end{equation}
where $C$ is a constant independent of $\sigma$.

\subsubsection{Nesterov's method}

Nesterov's accelerated gradient method for strongly convex optimization $\min\limits_{x} f(x)$ can be stated as follows~\cite{nesterov2003introductory}:
\begin{equation*}
    x^{k+1}=x^k+\beta(x^k-x^{k-1})-\alpha\nabla f(x^k+\beta(x^k-x^{k-1})),
\end{equation*}
where
    $\alpha = \frac{1}{L}$ and $\beta = \frac{\sqrt{L}-\sqrt{\mu}}{\sqrt{L}+\sqrt{\mu}}$.
Note that the above updating formula is for strongly convex minimization. (If the function is merely convex, we then let $\beta$ be depending on the iteration count $k$: $\beta_k=\frac{k}{k+3}$.)
Nesterov's method for strongly convex minimization has a R-linear convergence rate as follows:
\begin{equation*}
    f(x^k)-f(x^*)\leq 2(1-\sqrt{\sigma})^k\left(f(x^0)-f(x^*)\right).
\end{equation*}

In all the above methods, some {\it extra}\/ directions or points are introduced, in addition to the gradient direction. This motivates us to develop the proposed extra-point method.

\section{An Extra-Point Approach to Strongly Monotone VI}
\label{sec:vi}

Observe that in the previous subsections, the updating rules for the three methods take a similar form. First, we iteratively update an extrapolation point $x^{k+0.5}$ from the points $x^k$ and $x^{k-1}$, and the next iterative point is obtained from either $x^{k+0.5}$ or $x^{k}$ with a direction of either $\nabla f(x^k)$ or $\nabla f(x^{k+0.5})$. Specifically, assume $\mathcal{Z}=\mathbb{R}^n$ and assume $F(\cdot)$ is \textit{linear}, then the updates of OGDA can be written as:
\begin{equation*}
    \left\{
    \begin{array}{ccl}
         z^{k+0.5}&=&z^k+\beta(z^k-z^{k-1}), \\
         z^{k+1} &=& z^k-\alpha F(z^{k+0.5}).
    \end{array}
    \right.
\end{equation*}

The heavy-ball method is in the form:
\begin{equation*}
    \left\{
    \begin{array}{ccl}
         x^{k+0.5} &=& x^k+\beta(x^k-x^{k-1}), \\
         x^{k+1} &=& x^{k+0.5}-\alpha \nabla f(x^{k}).
    \end{array}
    \right.
\end{equation*}

Nesterov's accelerated method is in the form:
\begin{equation*}
    \left\{
    \begin{array}{ccl}
         x^{k+0.5} &=& x^k+\beta(x^k-x^{k-1}), \\
         x^{k+1} &=& x^{k+0.5}-\alpha \nabla f(x^{k+0.5}).
    \end{array}
    \right.
\end{equation*}

Obviously, the way how these terms/points are combined would matter for its practical performance, depending on the structure of the problem at hand. In this section,
we shall propose an inclusive general extra-point scheme to solve VI~\eqref{vi-prob}. Our aim is to provide a condition on the parameters ensuring the accelerated rate of convergence in the worst case.
We shall first consider the case where $F$ is defined on the whole space, and then extend the method to the case where the domain of $F$ is restricted to $\mathcal{Z}$.

\subsection{A general extra-point framework}
\label{sec:general-ex-point}

Our proposed extra-point scheme for solving the strongly monotone VI model~\eqref{vi-prob} is based on the following update formula:
\begin{equation}
    \label{update:extra-point}
    \left\{
    \begin{array}{ccl}
         z^{k+0.5} &=& z^k+\beta(z^k-z^{k-1})-\eta F(z^k), \\
         z^{k+1}   &=& P_{\mathcal{Z}}\left(z^k-\alpha F(z^{k+0.5})+\gamma(z^k-z^{k-1})-\tau\left(F(z^{k})-F(z^{k-1})\right)\right).
    \end{array}
    \right.
\end{equation}

The above procedure requires 5 nonnegative parameters to operate with: $\alpha,\beta,\gamma,\tau,\eta$; they represent respectively: extra-gradient, momentum, optimism, and gradient steps. These parameters need to be learned and fine-tuned to achieve good performances. Similar to the extra-gradient method \eqref{ex-grad-1}, we utilize the mapping $F(z^{k+0.5})$, where $z^{k+0.5}$ is defined differently as compared to \eqref{ex-grad-1}. Note that $z^{k+0.5}$ may be out of the constraint set $\mathcal{Z}$. However, through the projection operator for update $z^{k+1}$, the main sequence of iterative points $\{z^k\}$ for $k=0,1,2,...$ is entirely contained in $\mathcal{Z}$.

The following Table \ref{table:methods} displays how the five known algorithms manifest as special cases under this general scheme:

\begin{table}[h!]
\begin{tabular}{||l||c|c|c|c|c||l||}
\hline \hline
{\it  Existing Method  }      & $\alpha$ & $\beta$ & $\eta$ & $\gamma$ & $\tau$ & {\it The Dynamics } \\ \hline \hline
vanilla projection            & $+$      &  0      &   0    &    0     &   0    & $z^{k+1}=P_{\mathcal{Z}}\left(z^k-\alpha F(z^k)\right)$ \\ \hline
``heavy-ball''                & $+$      &  0      &   0    &   $+$    &   0    & $ z^{k+1}=P_{\mathcal{Z}}\left(z^k-\alpha F(z^k)+\gamma(z^k-z^{k-1})\right)$ \\ \hline
extra gradient                & $+$      &  0      &   $+$  &    0     &   0    & $z^{k+1}=P_{\mathcal{Z}}\left(z^k-\alpha F(z^k-\eta F(z^k))\right)$ \\ \hline
Nesterov's method             & $+$      &  $+$    &   0    &    $+$   &   0    & $z^{k+1}=P_{\mathcal{Z}}\left(z^k-\alpha F(z^k+\beta(z^k-z^{k-1}))+\gamma(z^k-z^{k-1})\right)$ \\ \hline
OGDA                          & $+$      &  0      &   0    &    0     &  $+$   & $z^{k+1}=P_{\mathcal{Z}}\left(z^k-\alpha F(z^{k})-\tau(F(z^{k})-F(z^{k-1}))\right)$ \\
\hline \hline
\end{tabular}
\caption{Parameters correspondence in different first-order methods}
\label{table:methods}
\end{table}

To analyze iteration complexity of the above scheme, let us first establish the following relation:

\begin{lemma}
\label{lem:extra-per-iter-conv}
For the sequences $\{z^k\}$ and $\{z^{k+0.5}\}$, $k=0,1,2,...$,  generated from the extra-point approach \eqref{update:extra-point}, the following inequality holds:
\begin{eqnarray}
        & & \|z^{k+1}-z^*\|^2 \nonumber \\
        & \leq & \left(1-\alpha\mu+3\gamma+\tau L(3+2\tau L+\frac{2\alpha}{\eta}+2\alpha L)+2\left(\gamma-\frac{\alpha\beta}{\eta}\right)^2+\left|-\frac{2\alpha\beta}{\eta}-\frac{2\alpha}{\eta}\left(\gamma-\frac{\alpha\beta}{\eta}\right)\right|\right)\|z^k-z^*\|^2 \nonumber \\
         &&+\left(2\left(\gamma-\frac{\alpha\beta}{\eta}\right)^2+\gamma+2\tau L(1+\tau L+\frac{\alpha}{\eta}+\alpha L)+\left|-\frac{2\alpha\beta}{\eta}-\frac{2\alpha}{\eta}\left(\gamma-\frac{\alpha\beta}{\eta}\right)\right|\right)\|z^{k-1}-z^*\|^2 \nonumber \\
         && + \left(\alpha^2L^2+\frac{\alpha^2}{\eta^2}+\frac{\alpha\tau L}{\eta}-\frac{2\alpha}{\eta}+2\alpha\mu+\alpha\tau L^2+\left|-\frac{\alpha\beta}{\eta}-\frac{\alpha}{\eta}\left(\gamma-\frac{\alpha\beta}{\eta}\right)\right|\right)\|z^k-z^{k+0.5}\|^2 \nonumber \\
         && + \left(-2\alpha+\frac{2\alpha^2}{\eta}\right)\left(F(z^{k+0.5})-F(z^k)\right)^{\top}(z^k-z^{k+0.5}) \nonumber \\
         && -2\alpha\left(\gamma-\frac{\alpha\beta}{\eta}\right)\left(F(z^{k+0.5})-F(z^k)\right)^{\top}(z^k-z^{k-1}) \nonumber \\
         && -2\tau\left(\gamma-\frac{\alpha\beta}{\eta}\right)\left(F(z^k)-F(z^{k-1})\right)^{\top}(z^k-z^{k-1}) . \label{relation-extra-points}
\end{eqnarray}

\end{lemma}
\begin{proof}
    See Appendix \ref{app:extra-per-iter-conv}.
\end{proof}

Lemma \ref{lem:extra-per-iter-conv} will provide a basis for our analysis, which relates the points $z^{k-1},z^k,z^{k+0.5},z^{k+1}$ generated by the extra-point scheme \eqref{update:extra-point}.
Towards a linear convergence rate, we require the parameters $\alpha,\beta,\gamma,\eta,\tau$ to fit the following patterns:
\begin{enumerate}
    \item There should only be $\|z^{k+1}-z^*\|^2$, $\|z^k-z^*\|^2$, $\|z^{k-1}-z^*\|^2$ in the relational bound~\eqref{relation-extra-points}. Therefore, the coefficient of $\|z^{k}-z^{k+0.5}\|^2$ should be non-positive; the coefficient of $(F(z^{k+0.5})-F(z^k))^{\top}(z^k-z^{k-1})$ should be 0; the coefficient of $(F(z^{k+0.5})-F(z^k))^{\top}(z^k-z^{k+0.5})$ should be non-negative; the coefficient of $\left(F(z^k)-F(z^{k-1})\right)^{\top}(z^k-z^{k-1})$ should be non-positive. Note that these terms themselves either non-positive or non-negative accordingly due to the monotonicity of $F$.
    \item A valid reduction rate should be established for the term $\|z^k-z^*\|^2$; that is:
    \begin{equation*}
        \alpha\mu-3\gamma-\tau L(3+2\tau L+\frac{2\alpha}{\eta}+2\alpha L)-2\left(\gamma-\frac{\alpha\beta}{\eta}\right)^2-\left|-\frac{2\alpha\beta}{\eta}-\frac{2\alpha}{\eta}\left(\gamma-\frac{\alpha\beta}{\eta}\right)\right|\in (0,1).
    \end{equation*}
    \item The coefficient of $\|z^{k-1}-z^*\|^2$ should not be too large. In particular, it should be smaller than the reduction rate for $\|z^k-z^*\|^2$:
    \begin{eqnarray*}
            0 &\leq & 2\left(\gamma-\frac{\alpha\beta}{\eta}\right)^2+\gamma+2\tau L(1+\tau L+\frac{\alpha}{\eta}+\alpha L)+\left|-\frac{2\alpha\beta}{\eta}-\frac{2\alpha}{\eta}\left(\gamma-\frac{\alpha\beta}{\eta}\right)\right|  \\
             & < & \alpha\mu-3\gamma-\tau L(3+2\tau L+\frac{2\alpha}{\eta}+2\alpha L)-2\left(\gamma-\frac{\alpha\beta}{\eta}\right)^2-\left|-\frac{2\alpha\beta}{\eta}-\frac{2\alpha}{\eta}\left(\gamma-\frac{\alpha\beta}{\eta}\right)\right|.
    \end{eqnarray*}
\end{enumerate}

Following the above three guidelines, \eqref{relation-extra-points} in Lemma \ref{lem:extra-per-iter-conv} leads to an inequality in the form of
\begin{equation}
    \label{lin-ab}
    \|z^{k+1}-z^*\|^2\leq (1-a)\|z^k-z^*\|^2+b\|z^{k-1}-z^*\|^2,
\end{equation}
where $0\leq b<a<1$ are constants to be determined.
The above bound can then be further transformed as follows:
\begin{lemma}
\label{lem:theta-range}
Suppose that \eqref{lin-ab} holds. Then, for any $\theta$ satisfying
\begin{equation*}
    b<\frac{\sqrt{(1-a)^2+4b}-(1-a)}{2}\leq \theta <a,
\end{equation*}
it holds that
\begin{equation*}
    \|z^{k+1}-z^*\|^2+\theta\|z^k-z^*\|^2\leq (1-(a-\theta))\|z^k-z^*\|^2+\theta(1-(a-\theta))\|z^{k-1}-z^*\|^2.
\end{equation*}
\end{lemma}
\begin{proof}
    The proof follows immediately from the fact that $b\leq\theta(1-(a-\theta))$.
\end{proof}

A specific choice of $\theta$ can be $\frac{a+b}{2}$. This leads to our main convergence result, summarized in the next theorem.

\begin{theorem}
\label{th:ex-point-result}
For solving VI model as in \eqref{vi-prob}, with $F:\mathbb{R}^n\rightarrow\mathbb{R}^n$ being Lipschitz continuous with constant $L$ and strongly monotone with constant $\mu$, the sequence $\{z^k\}$, $k=0,1,2,...$ generated from the extra-point approach \eqref{update:extra-point} satisfies 
\begin{equation*}
    \|z^{k+1}-z^*\|^2+\theta\|z^k-z^*\|^2\leq \left(1-(a-\theta)\right) \cdot \left(\|z^k-z^*\|^2+\theta \|z^{k-1}-z^*\|^2\right),
\end{equation*}
where
\begin{equation*}
    \begin{array}{ll}
         & a= \alpha\mu-3\gamma-\tau L(3+2\tau L+\frac{2\alpha}{\eta}+2\alpha L)-2\left(\gamma-\frac{\alpha\beta}{\eta}\right)^2-\left|-\frac{2\alpha\beta}{\eta}-\frac{2\alpha}{\eta}\left(\gamma-\frac{\alpha\beta}{\eta}\right)\right|,\\
         & b<\frac{\sqrt{(1-a)^2+4b}-(1-a)}{2}\leq \theta <a,  \\
         & b = 2\left(\gamma-\frac{\alpha\beta}{\eta}\right)^2+\gamma+2\tau L(1+\tau L+\frac{\alpha}{\eta}+\alpha L)+\left|-\frac{2\alpha\beta}{\eta}-\frac{2\alpha}{\eta}\left(\gamma-\frac{\alpha\beta}{\eta}\right)\right|,
    \end{array}
\end{equation*}
provided that the parameters $\alpha,\beta,\gamma,\eta,\tau$
satisfy the following constraints:
\begin{equation}
    \label{extra-point-const}
    \left\{
    \begin{array}{ll}
        \alpha\mu-4\gamma-\tau L(5+4\tau L+\frac{4\alpha}{\eta}+4\alpha L)-4\left(\gamma-\frac{\alpha\beta}{\eta}\right)^2-4\left|-\frac{\alpha\beta}{\eta}-\frac{\alpha\gamma}{\eta}+\frac{\alpha^2\beta}{\eta^2}\right|> 0,\\
         1> \alpha\mu-3\gamma-\tau L(3+2\tau L+\frac{2\alpha}{\eta}+2\alpha L)-2\left(\gamma-\frac{\alpha\beta}{\eta}\right)^2-\left|-\frac{2\alpha\beta}{\eta}-\frac{2\alpha}{\eta}\left(\gamma-\frac{\alpha\beta}{\eta}\right)\right|, \\
         \alpha^2L^2+\frac{\alpha^2}{\eta^2}+\frac{\alpha\tau L}{\eta}-\frac{2\alpha}{\eta}+2\alpha\mu+\alpha\tau L^2+\left|-\frac{\alpha\beta}{\eta}-\frac{\alpha}{\eta}\left(\gamma-\frac{\alpha\beta}{\eta}\right)\right|\leq 0, \\
          -2\alpha+\frac{2\alpha^2}{\eta} \geq 0, \\
         2\tau\left(\gamma-\frac{\alpha\beta}{\eta}\right)\geq 0, \\
          (\gamma\eta-\alpha\beta)\alpha=0, \\
          \alpha,\beta,\gamma,\tau\geq 0,\quad \eta >0.
    \end{array}
    \right.
\end{equation}
\end{theorem}
\begin{proof}
    The proof follows directly from Lemma \ref{lem:extra-per-iter-conv} and Lemma \ref{lem:theta-range}.
\end{proof}

Remark that \eqref{extra-point-const} can be easily satisfied. For example we may let
\begin{equation}
    \label{ex-point-par-example}
    (\alpha,\beta,\gamma,\eta,\tau)=\left(\frac{1}{4L},\frac{\sigma}{64},\frac{\sigma}{64},\frac{1}{4L},\frac{\sigma}{128L}\right),
\end{equation}
where $\sigma=\frac{\mu}{L}$, and we have
\begin{equation*}
    \frac{33\sigma}{256}>a>\frac{32\sigma}{256},\quad b<\frac{22\sigma}{256}.
\end{equation*}
Note that the convergence rate is R-linear with rate $1-(a-\theta)$. A simple estimation could be made by fixing $\theta=\frac{a+b}{2}$. With the choice in \eqref{ex-point-par-example}, the reduction rate is guaranteed to be at least $1-\frac{1}{2}\cdot\left(\frac{32\sigma}{256}-\frac{22\sigma}{256}\right)=1-\frac{5\sigma}{256}$. The resulting iteration complexity is therefore at the {\it optimal}\/ order of
 $   \mathcal{O}\left(\kappa 
 \ln\left(\frac{1}{\epsilon}\right)\right)$,
for finding an $\epsilon$-solution.

\subsection{The extra-point approach with domain restriction}
\label{sec:res-dom-ex-point}
In this subsection we shall present a variant of the extra-point approach \eqref{update:extra-point} shown in the previous subsection, in that the domain of $F(\cdot)$
is assumed to be $\mathcal{Z}$. In that case, the extrapolation point $z^{k+0.5}$ needs to be projected back to $\mathcal{Z}$, leading to the following updating formula:
\begin{equation}
    \label{update:extra-point-2}
    \left\{
    \begin{array}{ccl}
         z^{k+0.5} &=& P_{\mathcal{Z}}\left(z^k+\beta(z^k-z^{k-1})-\eta F(z^k)\right), \\
         z^{k+1} &=& P_{\mathcal{Z}}\left(z^k-\alpha F(z^{k+0.5})+\gamma(z^k-z^{k-1})-\tau\left(F(z^{k})-F(z^{k-1})\right)\right).
    \end{array}
    \right.
\end{equation}

Recall that in Section \ref{sec:proj-and-extra-grad} we also treated separately these two different settings for the extra-gradient method. There, the analysis for the case of restricted domain is different but the rate of convergence remains the same. For the extra-point approach, the modifications are subtler:
both the relations for the parameters and the reduction rate will need to be different, though the order of iteration complexity remains
$\mathcal{O}\left(\kappa
\ln\left(\frac{1}{\epsilon}\right)\right)$.

As in the previous section, we shall introduce below a relationship among the extra iterative points,
and the proof of the lemma will be relegated to the appendix.

\begin{lemma}
\label{lem:res-extra-per-iter-conv}
For the sequences $\{z^k\}$ and $\{z^{k+0.5}\}$, $k=0,1,2,...$  generated from the extra-point scheme \eqref{update:extra-point-2}, the following inequality holds
\begin{eqnarray}
    \label{exp-1-ineq-result}
         (1-\tau L)\|z^{k+1}-z^*\|^2& \leq & (1-\alpha\mu+4\gamma+2|\gamma-\beta|+2\tau L)\|z^k-z^*\|^2  \nonumber \\
         && + (2\gamma+2|\gamma-\beta|+2\tau L)\|z^{k-1}-z^*\|^2 \nonumber \\
         & & + (\alpha L +|\gamma-\beta|-1)\|z^{k+1}-z^{k+0.5}\|^2 \nonumber \\
         && +(\alpha L+2\alpha\mu+2\gamma-1)\|z^{k+0.5}-z^k\|^2  \nonumber \\
         && + 2(\eta-\alpha)F(z^k)^{\top}(z^{k+1}-z^{k+0.5}).
\end{eqnarray}
\end{lemma}
\begin{proof}
    See Appendix \ref{app:res-extra-per-iter-conv}.
\end{proof}

We make the following observations based on Lemma \ref{lem:res-extra-per-iter-conv}:
\begin{enumerate}
    \item To be able to rectify the bounds and relate only to the quantities $\|z^{k+1}-z^*\|^2$, $\|z^k-z^*\|^2$ and $\|z^{k-1}-z^*\|^2$, we need to let the coefficients of the terms $\|z^{k+1}-z^{k+0.5}\|^2$ and $\|z^{k+0.5}-z^k\|^2$ be non-positive, and let the coefficient of $F(z^k)^{\top}(z^{k+1}-z^{k+0.5})$ be 0.
    \item A valid linear reduction requires:
    \begin{equation*}
        0<\tau L<\alpha\mu-4\gamma-2|\gamma-\beta|-\tau L<1.
    \end{equation*}
    \item The coefficient of $\|z^{k-1}-z^*\|^2$ need be less than the difference between the coefficients of $\|z^{k+1}-z^*\|^2$ and $\|z^k-z^*\|^2$, i.e.
    \begin{equation*}
        2\gamma+2|\gamma-\beta|+2\tau L< (\alpha\mu-4\gamma-2|\gamma-\beta|-2\tau L)-\tau L.
    \end{equation*}
\end{enumerate}

The above three observations lead the relation in Lemma \ref{lem:res-extra-per-iter-conv} to take the form
\begin{equation*}
    (1-u)\|z^{k+1}-z^*\|^2\leq (1-s)\|z^k-z^*\|^2+t\|z^{k-1}-z^*\|^2,\quad 0\leq t\leq s-u<1-u.
\end{equation*}
Dividing both sides by $1-u$, we have
\begin{equation*}
    \|z^{k+1}-z^*\|^2\leq \left(1-\frac{s-u}{1-u}\right)\|z^k-z^*\|^2+\frac{t}{1-u}\|z^{k-1}-z^*\|^2,
\end{equation*}
which is in the form of \eqref{lin-ab} with
   $ a = \frac{s-u}{1-u},\quad b = \frac{t}{1-u}$.
By Lemma \ref{lem:theta-range}, we obtain a R-linear convergence, which takes a similar form as that in Theorem \ref{th:ex-point-result}. The result is summarized below.

\begin{theorem}
\label{th:ex-point-result-2}
For solving VI model \eqref{vi-prob}, with the mapping $F:\mathcal{Z}\rightarrow\mathbb{R}^n$ being Lipschitz continuous with constant $L$ and strongly monotone with constant $\mu$, the sequence $\{z^k\}$, $k=0,1,2,...$ generated from the extra-point approach \eqref{update:extra-point-2} yields a R-linear convergence as the following:
\begin{equation*}
    \|z^{k+1}-z^*\|^2+\theta\|z^k-z^*\|^2\leq (1-(a-\theta))\|z^k-z^*\|^2+\theta(1-(a-\theta))\|z^{k-1}-z^*\|^2,
\end{equation*}
where
\begin{equation*}
    \begin{array}{ll}
         & a= \frac{\alpha\mu+4\gamma+2|\gamma-\beta|-3\tau L}{1-\tau L},\\
         & b<\frac{\sqrt{(1-a)^2+4b}-(1-a)}{2}\leq \theta <a,\\
         & b = \frac{2\gamma+2|\gamma-\beta|+2\tau L}{1-\tau L}, \\
    \end{array}
\end{equation*}
provided that the parameters $\alpha,\beta,\gamma,\eta,\tau$ in~\eqref{update:extra-point-2} satisfy
\begin{equation}
    \label{exp-2-const}
    \left\{
    \begin{array}{ll}
         0\leq\tau L <  \alpha\mu-4\gamma-2|\gamma-\beta|-2\tau L<1,
         \\
         2\gamma+2|\gamma-\beta|+2\tau L< (\alpha\mu-4\gamma-2|\gamma-\beta|-2\tau L)-\tau L,
         \\
         \alpha L+|\gamma-\beta|-1\leq 0,\\
         \alpha L+2\alpha\mu+\tau L+2\gamma-1 \leq 0,\\
         \eta=\alpha, \\
         \alpha,\beta,\gamma,\eta,\tau \geq 0.
    \end{array}
    \right.
\end{equation}
\end{theorem}
\begin{proof}
    The proof follows directly from Lemma \ref{lem:res-extra-per-iter-conv} and Lemma \ref{lem:theta-range}.
\end{proof}

As an example, we may choose
    $\alpha=\eta=\frac{1}{4L}$, $\beta=\gamma=\frac{\mu}{64L}$, and $\tau=\frac{\mu}{64L^2}$
to satisfy \eqref{exp-2-const}. Then,
\begin{equation*}
    \left(1-\frac{\mu}{64L}\right)\|z^{k+1}-z^*\|^2\leq \left(1-\frac{5\mu}{32L}\right)\|z^k-z^*\|^2+\frac{\mu}{16L}\|z^{k-1}-z^*\|^2.
\end{equation*}
Denote $\frac{\mu}{L}=\sigma$ and divide both sides with $1-\frac{\sigma}{64}$, we get:
\begin{equation*}
    \begin{array}{lcl}
        \|z^{k+1}-z^*\|^2 & \leq & \left(1-\frac{\frac{9\sigma}{64}}{1-\frac{\sigma}{64}}\right)\|z^k-z^*\|^2+\frac{\frac{\sigma}{16}}{1-\frac{\sigma}{64}}\|z^{k-1}-z^*\|^2\\
         & = & \left(1-\frac{9\sigma}{64-\sigma}\right)\|z^k-z^*\|^2+\frac{\frac{\sigma}{16}}{1-\frac{\sigma}{64}}\|z^{k-1}-z^*\|^2 \\
         & \leq & \left(1-\frac{9\sigma}{64}\right)\|z^k-z^*\|^2+\frac{4\sigma}{63}\|z^{k-1}-z^*\|^2.
    \end{array}
\end{equation*}
Similarly, by fixing $\theta=\frac{a+b}{2}$, and with such choices of parameters, the reduction rate is guaranteed to be at least $1-\frac{1}{2}\cdot\left(\frac{9\sigma}{64}-\frac{4\sigma}{63}\right)<1-\frac{\sigma}{32}$. The resulting iteration complexity therefore reaches the {\it optimal} order of magnitude
 $   \mathcal{O}\left(\kappa 
 \ln\left(\frac{1}{\epsilon}\right)\right)$,
for finding an $\epsilon$-solution.

\section{An Extra-Point Approach to Strongly Convex Optimization}
\label{sec:opt}

Although convex optimization is a subclass of VI, its {\it optimal}\/ iteration complexity for convex optimization is also lower: it is $\mathcal{O}\left(\sqrt{\kappa}\, \ln \frac{1}{\epsilon}\right)$ instead of $\mathcal{O}\left(\kappa \ln \frac{1}{\epsilon}\right)$. Therefore, it requires a separate treatment if we wish to extend the general extra-point approach to strongly convex minimization with the corresponding optimal iteration complexity.

\subsection{Background and preparations}
\label{opt-lit}
One of the first methods that utilizes the concept of ``momentum'' may be traced back to the heavy-ball method \eqref{heavy-ball-update} proposed by Polyak in \cite{polyak1964some}. However, the accelerated convergence rate $1-\sqrt{\sigma}$ is only valid for a quadratic objective function, instead of general strongly convex function.
The first accelerated method for general convex optimization is Nesterov's method proposed in \cite{nesterov1983method} for general smooth convex functions.
The iteration complexity of Nesterov's accelerated method is $\mathcal{O}(1/\sqrt{\epsilon})$ for general convex optimization.
In \cite{nesterov2003introductory}, Nesterov introduced a variant of the method to solve
strongly convex optimization, with a $1-\sqrt{\sigma}$ linear convergence rate. Among its many equivalent forms, perhaps one of the most well-known form of Nesterov's method is the following updating formula
\begin{equation*}
    \left\{
    \begin{array}{ccl}
         y^k     &=&  \frac{1}{1+\sqrt{\sigma}}x^k+\frac{\sqrt{\sigma}}{1+\sqrt{\sigma}} v^k,  \\
         x^{k+1} &=& y^k-\frac{1}{L}\nabla f(y^k), \\
         v^{k+1} &=&  (1-\sqrt{\sigma}) v^k+\sqrt{\sigma}(y^k-\frac{1}{\mu}\nabla f(y^k)).
    \end{array}
    \right.
\end{equation*}
Choosing initial points appropriately, the above can be further reduced to the following updating rule:
\begin{equation}
    \label{Nes-reduced-form}
    \left\{
    \begin{array}{ccl}
         y^k &=& x^k+\frac{1-\sqrt{\sigma}}{1+\sqrt{\sigma}}(x^k-x^{k-1}), \\
         x^{k+1} &=& y^k-\frac{1}{L}\nabla f(y^k),
    \end{array}
    \right.
\end{equation}
which is the form that we used earlier.

In recent years, many other accelerated first-order methods have been proposed to achieve the $1-\sqrt{\sigma}$ convergence rate; some of them even have a better constant subsumed in the big O notation, compared to Nesterov's method. For example, the geometric descent method proposed in Bubeck {\it et al.}~\cite{bubeck2015geometric} updates some ball that contains the optimal solution $x^*$ throughout the iterations while trying to reduce its radius. The quadratic averaging method of Drusvyatskiy {\it et al.}~\cite{drusvyatskiy2018optimal} attempts to maximize the minimum of the convex combination of two quadratic lower bounds at each iteration, which actually produces the same iterative sequence as the geometric descent method.
The triple momentum method of Van Scoy {\it et al.}~\cite{van2017fastest} has the following general form:
\begin{equation*}
    \left\{
    \begin{array}{ccl}
         x^{k+1} &=& (1+\beta)x^k-\beta x^{k-1}-\alpha\nabla f(y^k), \\
         y^k &=& (1+\gamma)x^k-\gamma x^{k-1},\\
         v^k &=& (1+\delta)x^k-\delta x^{k-1}.
    \end{array}
    \right.
\end{equation*}
Although the method itself uses a different parameter choice of $\alpha,\beta,\gamma,\delta$, such general form can include heavy-ball method ($\gamma=\delta=0$) and Nesterov's fixed step size method ($\beta=\gamma$, $\delta=0$) as special cases. The information-theoretic exact method (ITEM) is recently proposed by Taylor and Drori~\cite{taylor2021optimal}, which has a similar form as Nesterov's method:
\begin{equation*}
    \left\{
    \begin{array}{ccl}
         y^k &=& (1-\tau_k)x^k+\tau_k v^k,  \\
         x^{k+1} &=& y^k-\frac{1}{L}\nabla f(y^k), \\
         v^{k+1}  &=&  (1-\sigma\delta_k) v^k+\delta_k(\sigma y^k-\frac{1}{L}\nabla f(y^k)).
    \end{array}
    \right.
\end{equation*}
However, different choices of the parameters $\delta_k,\tau_k$ and different potential function are used in the analysis. The ITEM achieves the exact lower bound in terms of the measurement $\frac{\|x^N-x^*\|^2}{\|x^0-x^*\|^2}$ for a given iteration number $N$ (see Drori and Taylor~\cite{drori2021oracle}). The method reduces to the triple momentum method by taking the parameters $\delta_k,\tau_k$ the values of their limits, and it reduces to the optimized gradient method of Kim and Fessler~\cite{kim2016optimized} by taking $\mu=0$.
The results in Karimi and Vavasis~\cite{karimi2016unified} establish a unified analysis for the conjugate gradient method and Nesterov's accelerated method, showing that the progress of both algorithms can be measured by the decrease of potential functions of the same form. A follow-up work by the same authors in \cite{karimi2017single} further includes the geometric descent in the analysis. In the recent monograph by d'Aspremont {\it et al.}~\cite{d2021acceleration}, an extensive survey is provided on the accelerated methods, including both convex and strongly convex optimization.

The classic reference on the information-theoretic results for lower iteration complexity bounds is Nemirovski and Yudin~\cite{nemirovsky1983problem}. In Nemirovski~\cite{nemirovsky1992information}, the lower complexity bound for convex quadratic optimization is established. Discussions on the iteration complexity lower bounds for general convex optimization $\mathcal{O}(1/\sqrt{\epsilon})$ and for strongly convex optimization $\mathcal{O}(\kappa\ln(1/\epsilon))$ can be found in Nesterov's monograph~\cite{nesterov2003introductory}. The most recent development regarding the iteration lower bound is in~\cite{drori2021oracle}, where the authors establish an exact lower bound for strongly convex optimization, in the sense that it is achieved up to a constant by ITEM.

Consider the following optimization model:
\begin{equation}
    \label{opt-prob}
    \min\limits_{x\in\mathbb{R}^n} f(x),
\end{equation}
where $f(x)$ is strongly convex with modulus $\mu$ and $\nabla f(x)$ is Lipschitz continuous with constant $L$. Therefore,  
\begin{equation*}
     f(x)+\nabla f(x)^{\top}(y-x)+\frac{\mu}{2}\|y-x\|^2 \le f(y) \leq f(x)+\nabla f(x)^{\top}(y-x)+\frac{L}{2}\|y-x\|^2,\,\,\, \forall x,y.
\end{equation*}
Recall that
    $\sigma = \frac{\mu}{L}$ and $\kappa = \frac{L}{\mu}$.
The vanilla gradient descent method with fixed step size has the following update:
\begin{equation*}
    x^{k+1} = x^k-\frac{1}{L}\nabla f(x^k).
\end{equation*}

We have
\begin{equation*}
    \begin{array}{lcl}
         f(x^{k+1}) & \leq & f(x^k)+\nabla f(x^k)^{\top}(x^{k+1}-x^k)+\frac{L}{2}\|x^{k+1}-x^k\|^2  \\
         &=& f(x^k)-\frac{1}{L}\|\nabla f(x^k)\|^2+\frac{1}{2L}\|\nabla f(x^k)\|^2 \\
         &= &f(x^k)-\frac{1}{2L}\|\nabla f(x^k)\|^2 \\
         &\leq & f(x)-\nabla f(x^k)^{\top}(x-x^k)-\frac{\mu}{2}\|x-x^k\|^2-\frac{1}{2L}\|\nabla f(x^k)\|^2,
    \end{array}
\end{equation*}
for any $x$.

Substitute $x^k$ and $x^*$ for $x$ in the above inequality, we get:
\begin{equation*}
    \left\{
    \begin{array}{ll}
         f(x^{k+1})\leq f(x^k)-\frac{1}{2L}\|\nabla f(x^k)\|^2, \\
         f(x^{k+1})\leq f(x^*)-\nabla f(x^k)^{\top}(x^*-x^k)-\frac{\mu}{2}\|x^*-x^k\|^2-\frac{1}{2L}\|\nabla f(x^k)\|^2.
    \end{array}
    \right.
\end{equation*}
Summing up the two inequalities with the first one multiplied by $1-\theta$ and the second one multiplied by $\theta$, we get:
\begin{equation}
    \label{grad-desc-conv}
    \begin{array}{lcl}
         f(x^{k+1})-f(x^*)&\leq & (1-\theta)(f(x^k)-f(x^*))+\theta\nabla f(x^k)^{\top}(x^k-x^*)-\frac{\mu\theta}{2}\|x^k-x^*\|^2-\frac{1}{2L}\|\nabla f(x^k)\|^2 \\
         & \leq  &(1-\theta)(f(x^k)-f(x^*)) +\frac{\theta}{2\mu}\|\nabla f(x^k)\|^2-\frac{1}{2L}\|\nabla f(x^k)\|^2 \\
         & \overset{\theta\leq \frac{\mu}{L}}{\leq} & (1-\theta)(f(x^k)-f(x^*)).
    \end{array}
\end{equation}

In the second inequality, the following bound is used:
\begin{equation*}
    \theta\nabla f(x^k)^{\top}(x^k-x^*)-\frac{\mu\theta}{2}\|x^k-x^*\|^2\leq \frac{\theta}{2\mu}\|\nabla f(x^k)\|^2.
\end{equation*}
Since $\theta\leq \frac{\mu}{L}=\sigma$, the per-iteration convergence rate is $1-\sigma$, giving an iteration complexity $\kappa\ln(1/\epsilon)$ in order to get $f(x^k)-f(x^*)\leq \epsilon$.

\subsection{An extra-point scheme}
\label{sec:opt-ex-point}
As shown in \eqref{grad-desc-conv}, directly bounding the term
\begin{equation}
    \label{opt-unwanted-terms}
    \theta\nabla f(x^k)^{\top}(x^k-x^*)-\frac{\mu\theta}{2}\|x^k-x^*\|^2-\frac{1}{2L}\|\nabla f(x^k)\|^2
\end{equation}
requires $\theta$ to be small enough, in particular, in the order of $\sigma$. This restricts the convergence rate to be $1-\sigma$ instead of the optimal $1-\sqrt{\sigma}$. However, this motivates us to introduce some extra point, from which the next iterate $x^{k+1}$ can be obtained from taking a gradient step there. In addition, instead of just taking $f(x^k)-f(x^*)$ as our measurement of the progress of the algorithm, we shall construct a potential function from an extra sequence of points $\{v^k\}$:
\begin{equation*}
    f(x^k)-f(x^*)+C\|v^k-x^*\|^2
\end{equation*}
for some constant $C$, specifically to help cancel out the terms similar to those in \eqref{opt-unwanted-terms} instead of directly bounding them.

Let us now consider the following iterative procedure
\begin{equation} \label{extra-points}
\left\{
\begin{array}{ccl}
p^k &:=& t_1 x^k + t_2 v^k \\
y^k & \mbox{satisfies} & \nabla f(y^k)^{\top}(p^k-y^k)\leq 0 \\
z^k &:=& y^k - \frac{t_3}{L} \nabla f(y^k) \\
x^{k+1} &:=& y^k - \frac{t_4}{L} \nabla f(z^k)-\frac{t_5}{L} \left(\nabla f(z^k)-\nabla f(y^k)\right)+t_6(z^k-y^k) \\
v^{k+1} &:=& t_7 v^k + t_8 y^k - t_9 \nabla f(y^k) ,
\end{array}
\right.
\end{equation}
where $t_i>0$, $i=1,...,9$, are some positive parameters. As examples, $y^k$ may be taken simply as $p^k$, or perhaps $y^k:=p^k-\frac{1}{L}\nabla f(p^k)$ among other choices.

To see how the extra points help in procedure \eqref{extra-points}, consider the following analysis. First, we have:
\begin{eqnarray} \label{ub1}
f(x^{k+1}) & \le & f(y^k) + \nabla f(y^k)^\top (x^{k+1}-y^k) + \frac{L}{2} \| x^{k+1} - y^k\|^2 \nonumber \\
&\le& f(x) - \nabla f(y^k)^\top (x-y^k) - \frac{\mu}{2} \| x - y^k\|^2 -\frac{t_4}{L} \nabla f(y^k)^\top \nabla f(z^k) + \frac{t_4^2}{2L} \| \nabla f(z^k)\|^2 \nonumber \\
&& -\frac{t_5}{L}\left(\nabla f(y^k)^{\top}\nabla f(z^k)-\|\nabla f(y^k)\|^2\right)-\frac{t_3t_6}{L}\|\nabla f(y^k)\|^2,
\end{eqnarray}
for any $x\in \mathbb{R}^n$. At the same time, let us set $0\le t_3<1$, and we derive
\begin{eqnarray}
\nabla f(y^k)^\top \nabla f(z^k) &=& \nabla f(y^k)^\top \left( \nabla f(y^k) + \nabla f(z^k) - \nabla f(y^k))\right) \nonumber \\
&\ge& \| \nabla f(y^k)\|^2 - \| \nabla f(y^k)\| \cdot L \|z^k-y^k\| \nonumber \\
&=& (1-t_3) \| \nabla f(y^k)\|^2 , \label{ub2}
\end{eqnarray}
and
\begin{eqnarray}
\| \nabla f(z^k)\|^2 &\le& \left( \| \nabla f(y^k)\| + \| \nabla f(z^k) - \nabla f(y^k) \| \right)^2 \nonumber \\
&\le& \left( \| \nabla f(y^k)\| + L \| z^k - y^k \| \right)^2 \nonumber \\
&=& (1+t_3)^2 \| \nabla f(y^k)\|^2. \label{ub3}
\end{eqnarray}
Summing up \eqref{ub2} and \eqref{ub3}, it follows from \eqref{ub1} that
\begin{equation} \label{ub4}
f(x^{k+1}) \le f(x) - \nabla f(y^k)^\top (x-y^k) - \frac{\mu}{2} \| x - y^k\|^2 - \frac{2t_4(1-t_3)-(1+t_3)^2t_4^2+2t_3(t_6-t_5)}{2L} \| \nabla f(y^k)\|^2 .
\end{equation}
Taking $x=x^k$ and $x=x^*$ respectively, where $x^*$ is the minimizer of $f$, we have
\begin{equation} \label{ineq1}
f(x^{k+1}) \le f(x^k) - \nabla f(y^k)^\top (x^k-y^k) - \frac{\mu}{2} \| x^k - y^k\|^2 - \frac{2t_4(1-t_3)-(1+t_3)^2t_4^2+2t_3(t_6-t_5)}{2L} \| \nabla f(y^k)\|^2
\end{equation}
and
\begin{equation} \label{ineq2}
f(x^{k+1}) \le f(x^*) - \nabla f(y^k)^\top (x^*-y^k) - \frac{\mu}{2} \| x^* - y^k\|^2 - \frac{2t_4(1-t_3)-(1+t_3)^2t_4^2+2t_3(t_6-t_5)}{2L} \| \nabla f(y^k)\|^2 .
\end{equation}
For a given $0<\theta <1$, let us multiply $1-\theta$ on both sides of \eqref{ineq1} and $\theta$ on both sides of \eqref{ineq2}, and then sum up the two inequalities. We obtain
\begin{eqnarray} \label{bound1}
&  & f(x^{k+1}) - f(x^*) \nonumber \\
&\le& (1-\theta) \left( f(x^k) - f(x^*)\right) + \nabla f(y^k)^\top \left[ (1-\theta) (y^k-x^k) + \theta (y^k - x^*)  \right]
 - \frac{\theta \mu}{2} \| x^* - y^k\|^2 \nonumber   \\
& & - \frac{\mu(1-\theta)}{2} \| x^k-y^k\|^2
- \frac{2t_4(1-t_3)-(1+t_3)^2t_4^2+2t_3(t_6-t_5)}{2L} \| \nabla f(y^k)\|^2 .
\end{eqnarray}

Now, instead of trying to directly bound the last four terms on the right hand side of \eqref{bound1}, the point $v^k$ comes to help cancel them out, as shown below.
Referring to \eqref{extra-points}, let us choose $t_7+t_8=1$. We have
\begin{eqnarray}
& & \| v^{k+1} - x^* \|^2 \nonumber \\
& \le & t_7 \| v^k - x^* \|^2 + t_8 \| y^k - x^*\|^2 + t_9^2 \| \nabla f(y^k)\|^2
 - 2 t_9 \nabla f(y^k)^\top \left[ t_7 v^k + t_8 y^k - x^* \right] . \label{bound2}
\end{eqnarray}
Selecting a parameter $C>0$, and summing \eqref{bound1} with \eqref{bound2} (multiplying $C$), we have
\begin{eqnarray}
& &  f(x^{k+1}) - f(x^*) + C \| v^{k+1} - x^* \|^2 \nonumber \\
&\le& (1-\theta) ( f(x^k) - f(x^*) ) + C t_7 \| v^k - x^* \|^2 \nonumber \\
& & + \left( C t_8 - \frac{\theta \mu}{2} \right)  \| y^k - x^*\|^2 + \left( C t_9^2 -  \frac{2t_4(1-t_3)-(1+t_3)^2t_4^2+2t_3(t_6-t_5)}{2L} \right)  \| \nabla f(y^k)\|^2 \nonumber \\
& & + \nabla f(y^k)^\top \left[  (1-\theta) (y^k-x^k) + \theta (y^k - x^*) - 2 t_9 C ( t_7 v^k + t_8 y^k - x^*) \right] .
\label{rate}
\end{eqnarray}

Now we shall choose the parameters so that the last three terms in \eqref{rate} are non-positive so that they can be dropped from the inequality. Summarizing, the requirements are:
\begin{equation}
\label{opt-ext-const}
\left\{
\begin{array}{ccl}
\theta &=& 2 t_9 C \\
t_1 &=& \frac{1-\theta}{1-2 t_8 t_9 C} \\
t_2 &=& \frac{2t_7 t_9 C}{1-2 t_8 t_9 C} \\
t_3 &<&1 \\
t_7 &\le& 1-\theta \\
t_8 &=& 1 - t_7 \\
t_8 C &\le& \frac{\mu\theta}{2} \\
t_9^2 C &\le & \frac{2t_4(1-t_3)-(1+t_3)^2t_4^2+2t_3(t_6-t_5)}{2L}.
\end{array}
\right.
\end{equation}

In particular, for the last term in \eqref{rate} we have:
\begin{eqnarray*}
        &&\nabla f(y^k)^\top \left[  (1-\theta) (y^k-x^k) + \theta (y^k - x^*) - 2 t_9 C ( t_7 v^k + t_8 y^k - x^*) \right] \\
        & = &\nabla f(y^k)^\top \left[  (1-2t_8t_9C)y^k-(1-\theta)x^k-2t_7t_9Cv^k \right]  \\
         & = &\nabla f(y^k)^\top \left[  (1-2t_8t_9C)y^k-(1-2t_8t_9C)p^k \right] \\
         & = &(1-2t_8t_9C)\nabla f(y^k)^{\top}(y^k-p^k)\leq 0.
\end{eqnarray*}

We summarize our findings and arrive at the following theorem:
\begin{theorem}
\label{th:opt-extra-1}
For an unconstrained optimization model \eqref{opt-prob} with the objective $f(x)$ being strongly convex with modulus $\mu>0$ and gradient Lipschitz with constant $L>0$, the sequence $\{x^k\}$, $k=0,1,2,...$ generated by the extra-point scheme \eqref{extra-points} converges linearly to the optimal solution:
\begin{equation*}
    f(x^{k+1})-f(x^*)+C\|v^{k+1}-x^*\|^2\leq (1-\theta)\left(f(x^k)-f(x^*)+C\|v^k-x^*\|^2\right)
\end{equation*}
for properly chosen $0<\theta<1$ and $C,t_i\geq 0$ for $i=1,2,..,9$ satisfying Condition~\eqref{opt-ext-const}.
\end{theorem}

We are left with the remaining final question: Whether or not Condition \eqref{opt-ext-const} is satisfiable at all? The answer is yes. For instance, we may choose any $0 < \delta < 1$ and then let
\begin{equation}
\label{opt-extra-par-choice}
\left\{
\begin{array}{l}
\theta = \sqrt{\frac{\mu}{L}} ,\,
t_1 = \frac{1}{1+\theta} ,\,
t_2 = \frac{\theta}{1+\theta} ,\,
t_3 = \delta ,\,
t_4 = \frac{1-\delta}{(1+\delta)^2} ,\,
t_5 = \frac{1}{(1+\delta)^2} ,\, \\
t_6 = \frac{3}{(1+\delta)^2} ,\,
t_7 = 1-\theta ,\,
t_8 = \theta ,\,
t_9 = \frac{1}{\sqrt{\mu L}} ,\,
C = \frac{\mu}{2} .
\end{array}
\right.
\end{equation}

If we further take $y^k = p^k$, procedure \eqref{extra-points} can be simplified to
\begin{equation} \label{extra-points-2}
\left\{
\begin{array}{ccl}
y^k &:=& \frac{ x^k + \theta v^k}{1+\theta} \\
z^k &:=& y^k - \frac{\delta}{L} \nabla f(y^k) \\
x^{k+1} &:=& y^k - \frac{1-\delta}{(1+\delta)^2 L} \nabla f(z^k)-\frac{1}{(1+\delta)^2L}\left(\nabla f(z^k)-\nabla f(y^k)\right)+\frac{3}{(1+\delta)^2}(z^k-y^k) \\
v^{k+1} &:=& (1-\theta) v^k + \frac{\theta(\mu\delta - L)}{\mu\delta} y^k +\frac{\theta L}{\mu \delta} z^k,
\end{array}
\right.
\end{equation}
and we have:
\begin{equation*}
    f(x^{k+1})-f(x^*)+\frac{\mu}{2}\|v^{k+1}-x^*\|^2\leq \left(1-\sqrt{\sigma}\right)\left(f(x^{k})-f(x^*)+\frac{\mu}{2}\|v^{k}-x^*\|^2\right).
\end{equation*}

The results are summarized in the next theorem:
\begin{theorem}
\label{th:opt-extra-2}
Following up on Theorem \ref{th:opt-extra-1}, if we further specify the parameters as in \eqref{opt-extra-par-choice}, then scheme \eqref{extra-points} reduces to \eqref{extra-points-2}. With the choice $v^0=x^0$, the sequence $\{x^k\}$ converges linearly to $x^*$ at the {\it optimal}\/ rate
\begin{equation*}
    f(x^{k})-f(x^*)\leq 2\left(1-\sqrt{\sigma}\right)^{k}\left(f(x^{0})-f(x^*)\right).
\end{equation*}
\end{theorem}

\section{Numerical Experiments}
\label{sec:numerical}

In this section, we use models with strongly monotone operators to test the performance of 
the proposed extra-point approach, to be compared with other existing accelerated first-order methods under both VI and optimization settings. We conduct two experiments under the VI setting, with the first one being unconstrained ($\mathcal{Z}=\mathbb{R}^n$) and the second being constrained by $z\ge0$ ($\mathcal{Z}=\mathbb{R}^n_{+}$), both with a {\it linear} operator $F(z)$. Note that in the unconstrained experiment we are equivalently solving a linear equation system $F(z^*)=0$, whereas in the constrained experiment the problem is equivalent to an LCP problem:
\[
z^*\ge 0,\quad F(z^*)\ge0,\quad (z^*)^\top F(z^*)=0.
\]
The linear strongly monotone operator $F(z)$ is set to be the same for both unconstrained and constrained cases, and is designed as the following:
\begin{eqnarray}
    \label{linear-operator}
    F(z) = Mz+q,
\end{eqnarray}
where $M=Q+A$ is the summation of a positive diagonal matrix $Q$ and a skew-symmetric matrix $A$, therefore a strongly monotone operator. The problem size in our experiment is $n=20$, and the matrices $Q,A,q$ are randomly generated. 
In particular, the diagonal elements of $Q$ are generated such that they have different orders of magnitude in size and the corresponding parameter $\sigma$ is in the order $10^{-2}$.

The parameters for each method are tuned to their best performance in a certain range. The specific numbers are shown in the following table, where the first number of each parameter corresponds to the unconstrained case and the second corresponds to the constrained case:
\begin{table}[h!]
\centering
{\footnotesize
\begin{tabular}{||l||c|c|c|c|c||} \hline \hline
{\it  First-order Method  }      & $\alpha$ & $\beta$ & $\eta$  & $\gamma$  & $\tau$  \\ \hline \hline
vanilla projection            & $0.0095$      / $0.0235$ &  0      &   0     &   0     &   0    \\ \hline
``heavy-ball''                & $0.0119$ / $0.0188$ &  0      &   0    &   $0.0365$ / $0.0146$    &   0      \\ \hline
extra-gradient                & $0.021$      / $0.034$ &  0        &   $0.021$  / $0.034$ &    0      &   0      \\ \hline
Nesterov's method             & $0.0084$      / $0.0146$ &  $0.175$     &   0     &    $0.175$    &   0      \\ \hline
OGDA                          & $0.019$      / $0.024$ &  0        &   0     &    0      &  $0.0117$ / $0.0234$    \\ \hline
extra-point method                           & $0.021$ / $0.034$     &  $0.3276$ / $0.34$      &   $0.0202$ / $0.0323$   &    $0.3276$ / $0.34$    &  $0.0021$ / $0.0068$   \\
\hline \hline
\end{tabular}
}
\caption{\small Parameter choices for different first-order methods for VI (unconstrained/constrained)}
\end{table}

We refer Table \ref{table:methods} for the parameters in the dynamics of each respective method. Note that in the constrained case, we use the domain-restricted variant for both the extra-gradient method \eqref{ex-grad-2} and the extra-point method \eqref{update:extra-point-2}. Figure \ref{fig:VI} shows the convergence behavior for the methods introduced in the above table. For the unconstrained case (the left plot in Figure \ref{fig:VI}), we use norm of the operator $\|F(z)\|$ as our measurement of convergence (merit function). On the other hand, we use $|z^\top F(z)|$ as our measurement of convergence for the constrained case (the right plot in Figure \ref{fig:VI}). Both measurements are displayed in log-scale.  In both experiments, the proposed extra-point approach has a superior convergence rate after fine-tuning the parameters, followed by extra-gradient and OGDA methods. The heavy-ball method, vanilla projection, and Nesterov's method have similar performance.

We also conduct two experiments under the optimization settings, with the first experiment having a {\it quadratic} objective function (i.e. a linear model) and the second having a non-quadratic objective function. Both models are {\it unconstrained}. In the quadratic case, we use the same linear form of \eqref{linear-operator} but with
\[
M = \sum\limits_{i=1}^nQ(i,i)*u_i*u_i^\top,
\]
where $n=20$ and $u_i$'s are randomly generated orthonormal vectors. The constant is $\sigma=0.0024$ for the experiment. Solving the equation system $F(z^*)=0$ is equivalent to solving a strongly convex quadratic minimization problem. In the non-quadratic case, we use the following {\it regularized logistic regression} model:
\[
\min\limits_{x}\frac{1}{N}\sum\limits_{i=1}^N\ln(1+e^{-a_i^\top x})+\frac{\lambda}{2}\|x\|^2.
\]
We set $N=2$, $\lambda=0.005$ in our experiment, where the problem size $n=15$ and $a_i$'s randomly generated. Note that although the objection function is strongly convex with modulus $\mu=\lambda=0.005$, the Lipschitz constant needs to be estimated for each of the method and it is accomplished by manually tuning the parameters.

In this experiment, we compare the extra-point method for optimization \eqref{extra-points} (setting $y^k=p^k$) with other first-order methods. The specific numbers for the parameters used are shown in Table \ref{table:quad-opt-par}. In the first table, the first number of each parameter corresponds to the quadratic case, whereas the second corresponds to the non-quadratic case. The parameters for the extra-point method are separated by rows as shown in the second table.
\begin{table}[h!]
\centering
{\small
\begin{tabular}{||l||c|c|c|c|c||} \hline \hline
{\it  First-order Method  }      & $\alpha$ & $\beta$& $\eta$  & $\gamma$  & $\tau$  \\ \hline \hline
gradient descent            & $0.0407$      / $38.4615$ &  0      &   0     &    0      &   0      \\ \hline
``heavy-ball''                & $0.0717$      / $9.8765$ &  0      &   0     &   $0.8349$ / $0.7778$    &   0      \\ \hline
extra-gradient                & $0.021$      / $19.7$ &  0        &   $0.021$  / $19.7$ &    0      &   0      \\ \hline
Nesterov's method             & $0.0214$      / $28.5714$ &  $0.9075$ / $0.455$    &   0     &    $0.9075$ / $0.455$   &   0      \\ \hline
OGDA                          & $0.0387$      / $39.2$ &  0        &   0     &    0      &  $0.002$ / $0.2$    \\
\hline \hline
\end{tabular}
}

\vspace{3mm}
{\small
\begin{tabular}{||l||c|c|c|c|c|c|c|c|c||}
\hline \hline
extra-point method & $t_1$ & $t_2$ & $t_3$& $t_4$& $t_5$& $t_6$& $t_7$& $t_8$& $t_9$\\ \hline \hline
quadratic  & $0.9538$ & $1-t_1$ & $0.9$& $0.0277$& $6.3712$& $6.9252$& $1-\sqrt{\sigma}$& $\sqrt{\sigma}$& $0.0485$\\\hline
non-quadratic & $0.7363$ & $1-t_1$ & $0.9$ & $0.0277$ & $5.5402$ & $6.6482$ & $0.6419$ & $1-t_7$ & $71.6115$\\
\hline
\hline
\end{tabular}
}
\caption{\small Parameter choices for different first-order methods for optimization (quadratic/non-quadratic)}
\label{table:quad-opt-par}
\end{table}

\begin{figure}[htbp]
\centering
\begin{minipage}[t]{0.48\textwidth}
\centering
\includegraphics[width=6cm]{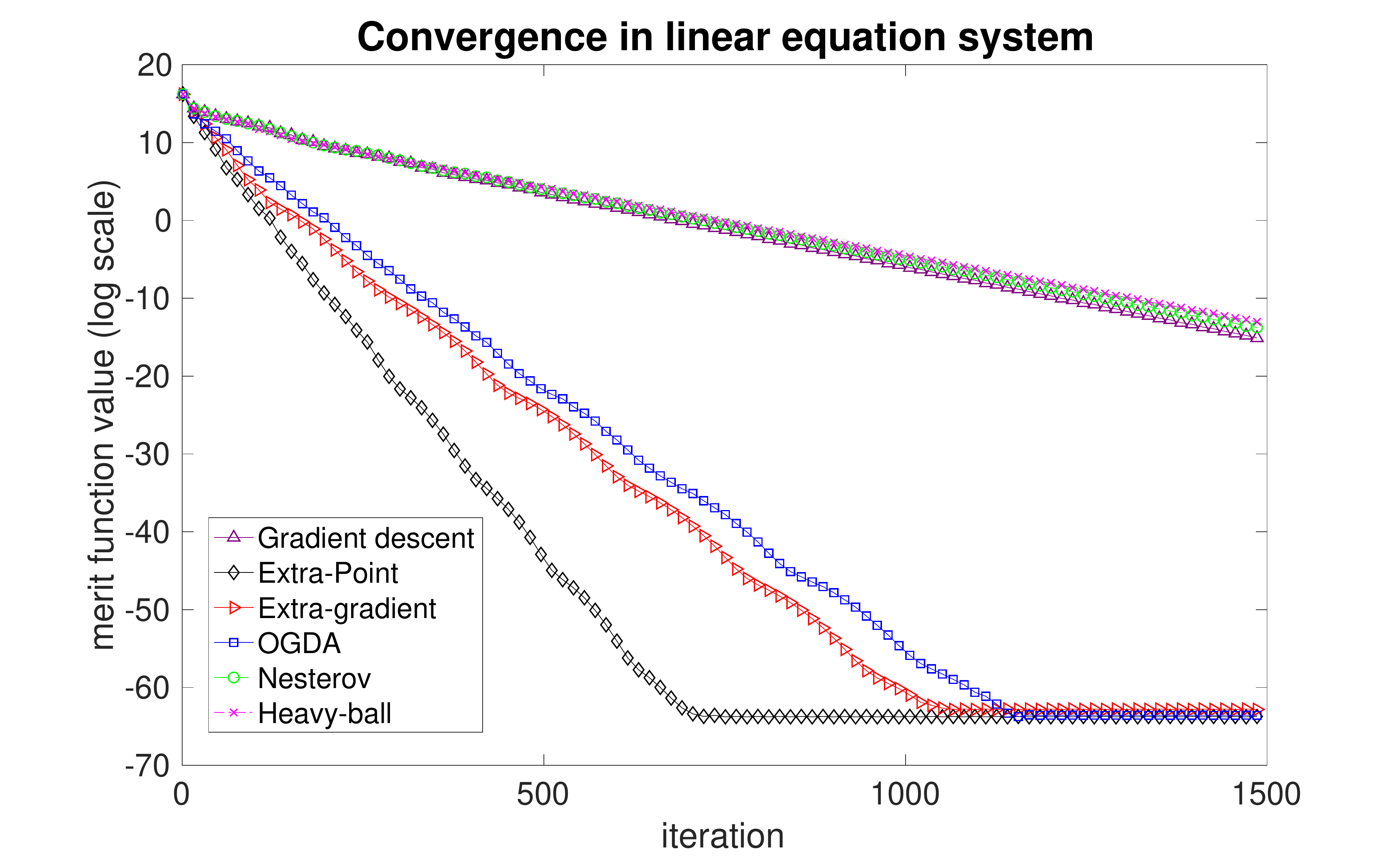}
\label{fig:VI-quad}
\end{minipage}
\begin{minipage}[t]{0.48\textwidth}
\centering
\includegraphics[width=6cm]{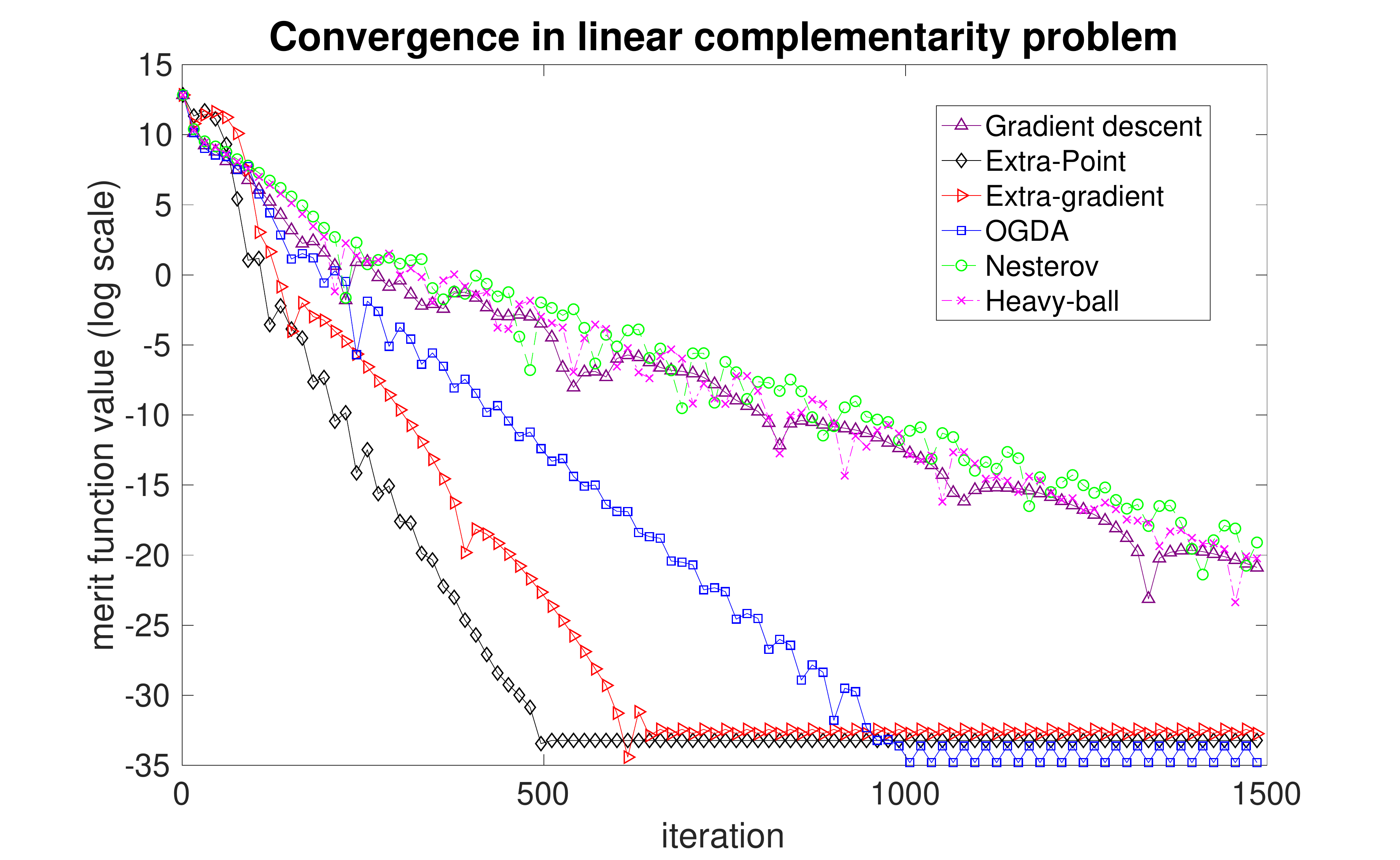}
\label{fig:opt-quad}
\end{minipage}
\caption{Convergence with strongly monotone operator}
\label{fig:VI}
\end{figure}
\begin{figure}[htbp]
\centering
\begin{minipage}[t]{0.48\textwidth}
\centering
\includegraphics[width=6cm]{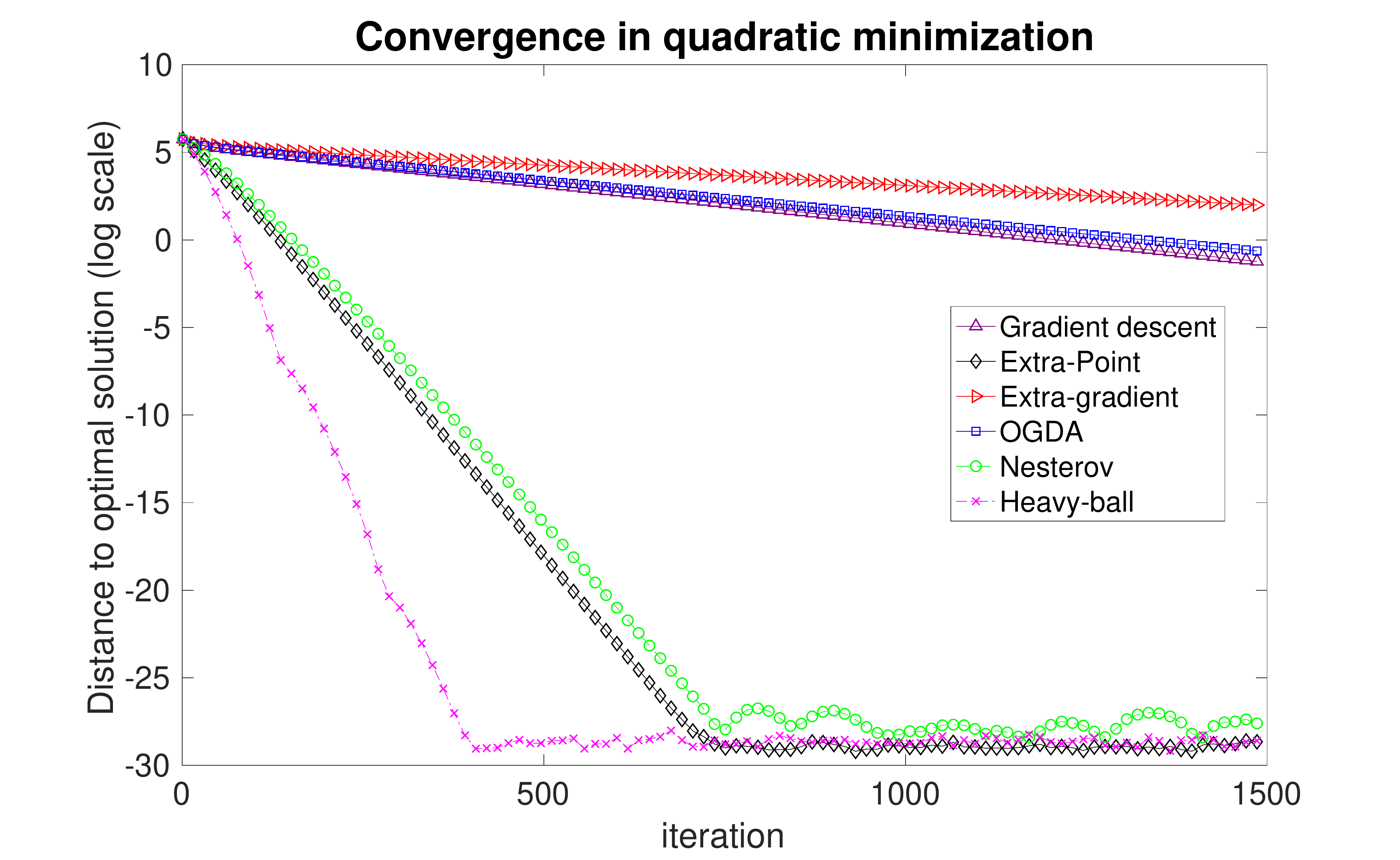}
\label{fig:VI-log}
\end{minipage}
\begin{minipage}[t]{0.48\textwidth}
\centering
\includegraphics[width=6cm]{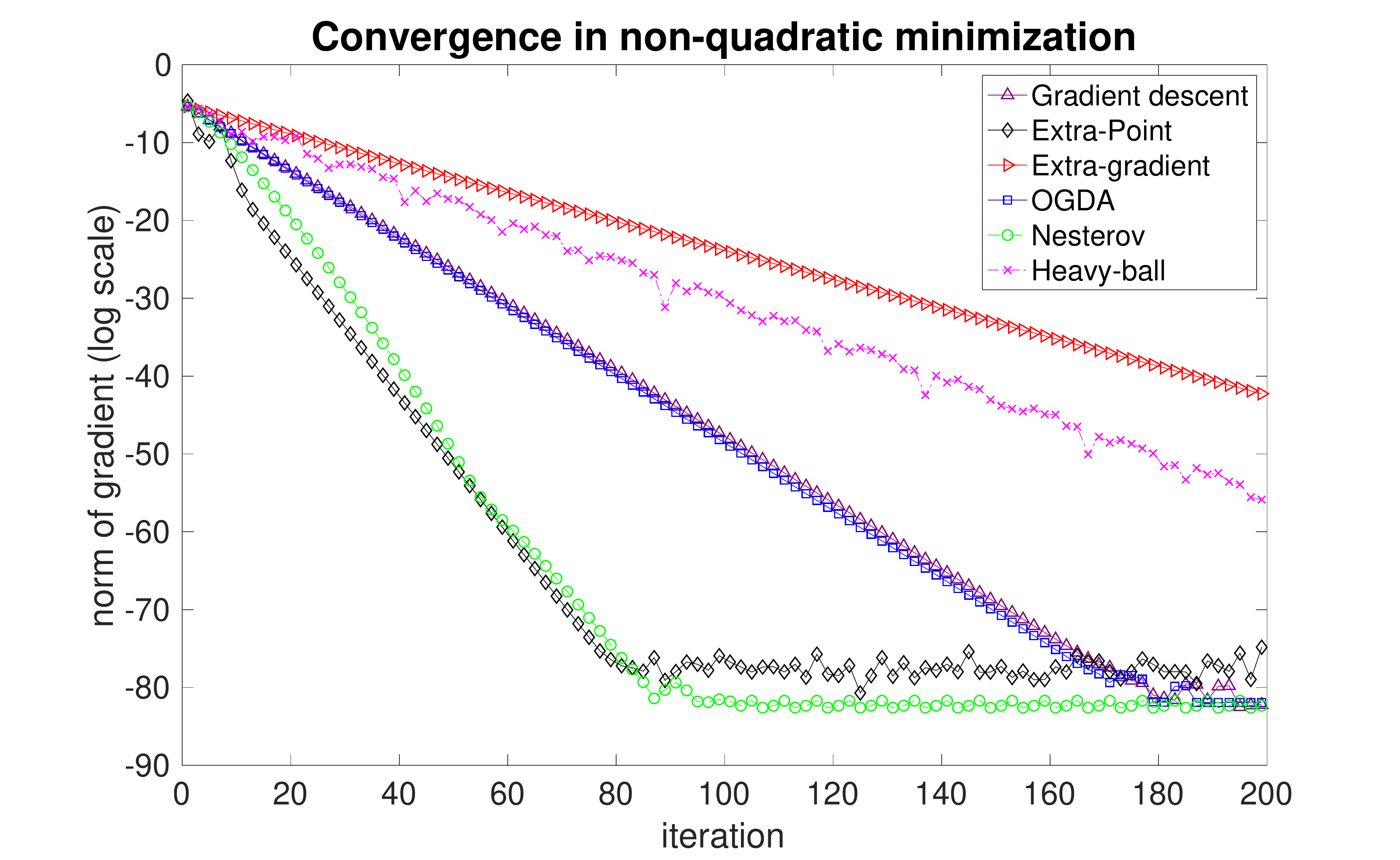}
\label{fig:opt-log}
\end{minipage}
\caption{Convergence in strongly convex optimization }
\label{fig:opt}
\end{figure}

The convergence results for each method are shown in Figure \ref{fig:opt}. The progress is measured as the distance to the optimal solution (from solving the linear equation system directly) for the quadratic case, and is measured as norm of gradient for the non-quadratic case. Under the quadratic setting (the left plot in Figure \ref{fig:opt}), the heavy-ball method has the best performance, followed by Nesterov's method and the extra-point method.  
Other than the afore-mentioned three methods which demonstrate clear acceleration, the extra-gradient method, gradient descent, and OGDA have similar slower convergence. Under the non-quadratic setting (the right plot in Figure \ref{fig:opt}), Nesterov's method and the extra-point method have similar leading performance (note that the parameter choices for both methods are quite different), followed by gradient descent and OGDA. The heavy-ball method and the extra-gradient method demonstrate slower convergence in this experiment.

\section{Conclusion}
\label{sec:conclusion}
In this paper, we propose a unifying framework of accelerated first-order methods for solving strongly monotone VI. The proposed extra-point approach is an inclusive framework that can specialize to the extra-gradient method, OGDA, Nesterov's accelerated method, and the heavy-ball method, with proper parameter choices. We also extend the framework to the context of strongly convex optimization.
By means of deriving the worst-case iteration bounds
and conducting preliminary numerical experiments,
we show that acceleration can be made rather flexible, not necessarily restricting to the exact form of any specific methods, such as extra-gradient method for VI or the Nesterov-style extrapolations for optimization. Other first-order directions/ideas such as ``heavy-ball'', ``optimism'', etc., may contribute their fair shares to acceleration in practice too. For achieving superior numerical performances, the question remains: How do we find appropriate doses for these terms in a ``cocktail'' implementation? We believe that the answer might be: {\it It depends}. It depends on the structure of the problem at hand. A reasonable approach would be to {\it learn}\/ a good combination of the terms through experimenting with the training data for the given class of problems at hand. In this sense, optimization helps machine learning, while machine learning can also help enhancing optimization in return. After all, to provide a flexible ground for the learning to be possible is the primary purpose behind the proposed new scheme.

\printbibliography

\begin{appendices}
\section{Proofs of Propositions and Theorems}
\subsection{Proof of Theorem \ref{th:ogda}}
\label{app:ogda}

Let us derive the chain of inequalities, where the underlined terms highlight some changes between equalities/inequalities for easier following:
\begin{eqnarray*}
         & & \|z^{k+1}-z^*\|^2 \\
         & = & \|P_{\mathcal{Z}}\left(z^k-\alpha F(z^k)-\tau(F(z^k)-F(z^{k-1}))\right)-P_{\mathcal{Z}}\left(z^*-\alpha F(z^*)\right)\|^2  \\
         & \leq & (z^{k+1}-z^*)^{\top}(z^k-\alpha F(z^k)-\tau(F(z^k)-F(z^{k-1}))-z^*+\alpha F(z^*)) \\
         & = & (z^{k+1}-z^*)^{\top}(z^k-z^*)-\alpha(z^{k+1}-z^*)^{\top}(F(z^k)-F(z^*))-\tau(z^{k+1}-z^*)^{\top}(F(z^k)-F(z^{k-1})) \\
         & = & \frac{1}{2}\left(\|z^{k+1}-z^*\|^2+\|z^k-z^*\|^2-\|z^{k+1}-z^k\|^2\right)-\tau(z^{k+1}-z^*)^{\top}(F(z^k)-F(z^{k-1})) \\
         && -\alpha(z^{k+1}-z^*)^{\top}(F(z^k)-F(z^{k+1}))-\underline{\alpha(z^{k+1}-z^*)^{\top}(F(z^{k+1})-F(z^*))}\\
         & \leq & \frac{1}{2}\left(\|z^{k+1}-z^*\|^2+\|z^k-z^*\|^2-\|z^{k+1}-z^k\|^2\right)-\tau(z^{k+1}-z^*)^{\top}(F(z^k)-F(z^{k-1})) \\
         && -\alpha(z^{k+1}-z^*)^{\top}(F(z^k)-F(z^{k+1}))-\underline{\alpha\mu\|z^{k+1}-z^*\|^2} \\
         & = & \frac{1}{2}\left(\|z^{k+1}-z^*\|^2+\|z^k-z^*\|^2-\|z^{k+1}-z^k\|^2\right) \\
         && -\alpha(z^{k+1}-z^*)^{\top}(F(z^k)-F(z^{k+1}))-\alpha\mu\|z^{k+1}-z^*\|^2 \\
         &&-\underline{\tau(z^{k+1}-z^k)^{\top}(F(z^k)-F(z^{k-1}))}-\tau(z^{k}-z^*)^{\top}(F(z^k)-F(z^{k-1}))\\
         & \leq & \frac{1}{2}\left(\|z^{k+1}-z^*\|^2+\|z^k-z^*\|^2-\|z^{k+1}-z^k\|^2\right) \\
         && -\alpha(z^{k+1}-z^*)^{\top}(F(z^k)-F(z^{k+1}))-\alpha\mu\|z^{k+1}-z^*\|^2 \\
         &&+\underline{\frac{L\tau}{2}\|z^{k+1}-z^k\|^2+\frac{L\tau}{2}\|z^k-z^{k-1}\|^2}-\tau(z^{k}-z^*)^{\top}(F(z^k)-F(z^{k-1}))\\
         & \leq & \frac{1}{2}\left(\|z^{k+1}-z^*\|^2+\|z^k-z^*\|^2-\|z^{k+1}-z^k\|^2\right) \\
         && -\alpha(z^{k+1}-z^*)^{\top}(F(z^k)-F(z^{k+1}))-\alpha\mu\|z^{k+1}-z^*\|^2 .
\end{eqnarray*}

Rearranging the above (last) inequality we get:
\begin{eqnarray}
    \label{ogda-heavy-result-1}
         & & \left(\frac{1}{2}+\alpha\mu\right)\|z^{k+1}-z^*\|^2+\alpha(z^{k+1}-z^*)^{\top}(F(z^k)-F(z^{k+1}))+\left(\frac{1}{2}-\frac{L\tau}{2}\right)\|z^{k+1}-z^k\|^2 \nonumber \\
         &\leq&  \frac{1}{2}\|z^k-z^*\|^2+\tau(z^k-z^*)^{\top}(F(z^{k-1})-F(z^k))+\frac{L\tau}{2}\|z^k-z^{k-1}\|^2.
\end{eqnarray}
One can see that in the above inequality, the iteration counts of the three terms on the LHS and the iteration counts of the three terms on the RHS differ exactly by one, repsctively. The convergence then rely on the ratio between each corresponding term on the LHS/RHS by choosing the parameters appropriately.

Let us choose
   $ \alpha=\frac{1}{2L}$ and  $\tau=\frac{\alpha}{1+\sigma}=\frac{1}{2L}\cdot \frac{1}{1+\sigma}$.
Then we have:
\begin{eqnarray}
    \label{ogda-heavy-result-2}
& & \left(\frac{1}{2}+\frac{\sigma}{2}\right)\|z^{k+1}-z^*\|^2+\frac{1}{2L}(z^{k+1}-z^*)^{\top}(F(z^k)-F(z^{k+1}))+\left(\frac{1}{2}-\frac{1}{4}\cdot\frac{1}{1+\sigma}\right)\|z^{k+1}-z^k\|^2 \nonumber \\
         &\overset{\eqref{ogda-heavy-result-1}}{\leq}& \frac{1}{2}\|z^k-z^*\|^2+\frac{1}{2L}\cdot\frac{1}{1+\sigma}(z^k-z^*)^{\top}(F(z^{k-1})-F(z^k))+\underline{\frac{1}{4}\cdot\frac{1}{1+\sigma}}\|z^k-z^{k-1}\|^2 \nonumber \\
         &\leq&  \frac{1}{2}\|z^k-z^*\|^2+\frac{1}{2L}\cdot\frac{1}{1+\sigma}(z^k-z^*)^{\top}(F(z^{k-1})-F(z^k))+\underline{\left(\frac{1}{2}-\frac{1}{4}\cdot\frac{1}{1+\sigma}\right)\cdot\frac{1}{1+\sigma}}\|z^k-z^{k-1}\|^2 \nonumber \\
         &=& \frac{1}{2}V_k.
\end{eqnarray}

Observe that the LHS of \eqref{ogda-heavy-result-2} is exactly $(1+\sigma)$ times that of the RHS.
Thus, we have
\begin{equation*}
    \left(1+\sigma\right)V_{k+1}\leq V_k.
\end{equation*}

Finally, by noting the second term of $V_k$ can be bounded by the following:
\begin{eqnarray*}
         (z^k-z^*)^{\top}(F(z^{k-1})-F(z^k)) & \geq & -L\|z^k-z^*\|\cdot\|z^{k-1}-z^*\|  \\
         & \geq & -\frac{L}{2}\|z^k-z^*\|^2-\frac{L}{2}\|z^{k-1}-z^k\|^2,
\end{eqnarray*}
we have
\begin{equation*}
    V_k\geq \frac{1}{2}\|z^k-z^*\|^2,
\end{equation*}
which eventually gives us:
\begin{equation*}
    \frac{1}{2}\|z^k-z^*\|^2\leq V_k\leq \left(1+\sigma\right)^{-k}V_0=\left(1+\sigma\right)^{-k}\|z^0-z^*\|^2.
\end{equation*}

\subsection{Proof of Lemma \ref{lem:extra-per-iter-conv}}
\label{app:extra-per-iter-conv}

We derive the chain of inequalities (the underlined terms highlight the changes between equalities/inequalities to help follow the derivation):
\begin{eqnarray*}
         & & \|z^{k+1}-z^*\|^2 \\
         & = & \|P_{\mathcal{Z}}\left(z^k-\alpha F(z^{k+0.5})+\gamma(z^k-z^{k-1})-\tau(F(z^{k})-F(z^{k-1}))\right)-P_{\mathcal{Z}}(z^*)\|^2  \\
         & \leq & \|z^k-\alpha F(z^{k+0.5})+\gamma(z^k-z^{k-1})-\tau(F(z^{k})-F(z^{k-1}))-z^*\|^2 \\
         & = & \|z^k-z^*-\alpha (F(z^{k+0.5})-F(z^k)+\underline{F(z^k)})+\gamma(z^k-z^{k-1})-\tau(F(z^{k})-F(z^{k-1}))\|^2 \\
         & = & \|z^k-z^*-\alpha (F(z^{k+0.5})-F(z^k))-\underline{\frac{\alpha}{\eta}(z^k-z^{k+0.5}+\beta(z^k-z^{k-1}))}+\gamma(z^k-z^{k-1}) \\
         && -\tau(F(z^{k})-F(z^{k-1}))\|^2 \\
         & = & \|z^k-z^*-\alpha (F(z^{k+0.5})-F(z^k))-\frac{\alpha}{\eta}(z^k-z^{k+0.5})+\left(\gamma-\frac{\alpha\beta}{\eta}\right)(z^k-z^{k-1})-\tau(F(z^{k})-F(z^{k-1}))\|^2\\
         & = & \|z^k-z^*\|^2+\underline{\alpha^2\|F(z^{k+0.5})-F(z^k)\|^2+\frac{\alpha^2}{\eta^2}\|z^k-z^{k+0.5}\|^2}+\underline{\left(\gamma-\frac{\alpha\beta}{\eta}\right)^2\|z^k-z^{k-1}\|^2} \\
         && +\underline{\tau^2\|F(z^k)-F(z^{k-1})\|^2} \\
         && +\underline{2\gamma(z^k-z^*)^{\top}(z^k-z^{k-1})}-\underline{2\alpha(z^k-z^*)^{\top}F(z^{k+0.5})}-\underline{2\tau(z^k-z^*)^{\top}(F(z^{k})-F(z^{k-1}))}\\
         & & +\underline{2\alpha\tau\left(F(z^{k+0.5})-F(z^k)\right)^{\top}\left(F(z^k)-F(z^{k-1})\right)}+\frac{2\alpha^2}{\eta}(F(z^{k+0.5})-F(z^k))^{\top}(z^k-z^{k+0.5}) \\
         & & -2\alpha\left(\gamma-\frac{\alpha\beta}{\eta}\right)(F(z^{k+0.5})-F(z^k))^{\top}(z^k-z^{k-1})-\frac{2\alpha}{\eta}\left(\gamma-\frac{\alpha\beta}{\eta}\right)(z^k-z^{k+0.5})^{\top}(z^k-z^{k-1}) \\
         && +\underline{\frac{2\alpha\tau}{\eta}\left(F(z^k)-F(z^{k-1})\right)^{\top}(z^k-z^{k+0.5})}-2\tau\left(\gamma-\frac{\alpha\beta}{\eta}\right)\left(F(z^k)-F(z^{k-1})\right)^{\top}(z^k-z^{k-1}) \\
         & \leq & \|z^k-z^*\|^2+\underline{\alpha^2\left(L^2+\frac{1}{\eta^2}\right)\|z^k-z^{k+0.5}\|^2}+\underline{\left(\tau^2L^2+\left(\gamma-\frac{\alpha\beta}{\eta}\right)^2\right)\|z^k-z^{k-1}\|^2} \\
         && +\underline{\gamma\left(\|z^k-z^*\|^2+\|z^k-z^{k-1}\|^2-\|z^{k-1}-z^*\|^2\right)}+\underline{\tau L\left(\|z^k-z^*\|^2+\|z^k-z^{k-1}\|^2\right)}\\
         && \underline{-2\alpha(z^k-z^{k+0.5})^{\top}F(z^{k+0.5})-\alpha\mu\|z^k-z^*\|^2+2\alpha\mu\|z^k-z^{k+0.5}\|^2} \\
         & & +\underline{\alpha\tau L^2\left(\|z^{k+0.5}-z^k\|^2+\|z^k-z^{k-1}\|^2\right)}+\frac{2\alpha^2}{\eta}(F(z^{k+0.5})-F(z^k))^{\top}(z^k-z^{k+0.5}) \\
         & & -2\alpha\left(\gamma-\frac{\alpha\beta}{\eta}\right)(F(z^{k+0.5})-F(z^k))^{\top}(z^k-z^{k-1})-\frac{2\alpha}{\eta}\left(\gamma-\frac{\alpha\beta}{\eta}\right)(z^k-z^{k+0.5})^{\top}(z^k-z^{k-1}) \\
         && + \underline{\frac{\alpha\tau L}{\eta}
\left(\|z^k-z^{k-1}\|^2+\|z^k-z^{k+0.5}\|^2\right)}-2\tau\left(\gamma-\frac{\alpha\beta}{\eta}\right)\left(F(z^k)-F(z^{k-1})\right)^{\top}(z^k-z^{k-1}) .
\end{eqnarray*}
Using
\begin{equation*}
    F(z^{k+1/2})=F(z^{k+1/2})-F(z^k)+\frac{1}{\eta}(z^k-z^{k+1/2}+\beta(z^k-z^{k-1})),
\end{equation*}
we have:
\begin{eqnarray*}
    & & -2\alpha(z^{k+0.5}-z^k)^{\top}F(z^{k+0.5}) \\
    &=& -2\alpha(z^k-z^{k+0.5})^{\top}\left[F(z^{k+0.5})-F(z^k)+\frac{1}{\eta}(z^k-z^{k+0.5}+\beta(z^k-z^{k-1}))\right] .
\end{eqnarray*}

Combining and rearranging we obtain
\begin{eqnarray*}
        & & \|z^{k+1}-z^*\|^2 \\ & \leq & (1-\alpha\mu+\gamma+\tau L)\|z^k-z^*\|^2-\gamma\|z^{k-1}-z^*\|^2 \\
         && + \left(\alpha^2L^2+\frac{\alpha^2}{\eta^2}+\frac{\alpha\tau L}{\eta}-\frac{2\alpha}{\eta}+2\alpha\mu+\alpha\tau L^2\right)\|z^k-z^{k+0.5}\|^2 \\
         &&+ \left(\tau^2L^2+\left(\gamma-\frac{\alpha\beta}{\eta}\right)^2+\gamma+\tau L+\frac{\alpha\tau L}{\eta}+\alpha\tau L^2\right)\|z^k-z^{k-1}\|^2 \\
         && + \left(-2\alpha+\frac{2\alpha^2}{\eta}\right)(z^k-z^{k+0.5})^{\top}\left(F(z^{k+0.5})-F(z^k)\right) \\
         &&+ \left(-\frac{2\alpha\beta}{\eta}-\frac{2\alpha}{\eta}\left(\gamma-\frac{\alpha\beta}{\eta}\right)\right)(z^k-z^{k+0.5})^{\top}(z^k-z^{k-1}) \\
         && -2\alpha\left(\gamma-\frac{\alpha\beta}{\eta}\right)(F(z^{k+0.5})-F(z^k))^{\top}(z^k-z^{k-1}) \\
         && -2\tau\left(\gamma-\frac{\alpha\beta}{\eta}\right)\left(F(z^k)-F(z^{k-1})\right)^{\top}(z^k-z^{k-1}) .
\end{eqnarray*}

With the inequality
\begin{equation*}
    2(z^k-z^{k+0.5})^{\top}(z^k-z^{k-1})\leq \|z^k-z^{k+0.5}\|^2+\|z^k-z^{k-1}\|^2,
\end{equation*}
the above bound further reduces to:
\begin{eqnarray*}
        & & \|z^{k+1}-z^*\|^2 \\ & \leq & (1-\alpha\mu+\gamma+\tau L)\|z^k-z^*\|^2-\gamma\|z^{k-1}-z^*\|^2 \\
         && + \left(\alpha^2L^2+\frac{\alpha^2}{\eta^2}+\frac{\alpha\tau L}{\eta}-\frac{2\alpha}{\eta}+2\alpha\mu+\alpha\tau L^2+\left|-\frac{\alpha\beta}{\eta}-\frac{\alpha}{\eta}\left(\gamma-\frac{\alpha\beta}{\eta}\right)\right|\right)\|z^k-z^{k+0.5}\|^2 \\
         &&+ \left(\tau^2L^2+\left(\gamma-\frac{\alpha\beta}{\eta}\right)^2+\gamma+\tau L+\frac{\alpha\tau L}{\eta}+\alpha\tau L^2+\left|-\frac{\alpha\beta}{\eta}-\frac{\alpha}{\eta}\left(\gamma-\frac{\alpha\beta}{\eta}\right)\right|\right)\|z^k-z^{k-1}\|^2 \\
         && + \left(-2\alpha+\frac{2\alpha^2}{\eta}\right)(z^k-z^{k+0.5})^{\top}\left(F(z^{k+0.5})-F(z^k)\right) \\
         && -2\alpha\left(\gamma-\frac{\alpha\beta}{\eta}\right)(F(z^{k+0.5})-F(z^k))^{\top}(z^k-z^{k-1}) \\
         && -2\tau\left(\gamma-\frac{\alpha\beta}{\eta}\right)\left(F(z^k)-F(z^{k-1})\right)^{\top}(z^k-z^{k-1}) .
\end{eqnarray*}

Finally, with
\begin{equation*}
    \|z^k-z^{k-1}\|^2\leq 2\|z^k-z^*\|^2+2\|z^{k-1}-z^*\|^2,
\end{equation*}
the bound reduces to
\begin{eqnarray*}
        & & \|z^{k+1}-z^*\|^2 \\
        & \leq & \left(1-\alpha\mu+3\gamma+\tau L(3+2\tau L+\frac{2\alpha}{\eta}+2\alpha L)+2\left(\gamma-\frac{\alpha\beta}{\eta}\right)^2+\left|-\frac{2\alpha\beta}{\eta}-\frac{2\alpha}{\eta}\left(\gamma-\frac{\alpha\beta}{\eta}\right)\right|\right)\|z^k-z^*\|^2 \\
         &&+\left(2\left(\gamma-\frac{\alpha\beta}{\eta}\right)^2+\gamma+2\tau L(1+\tau L+\frac{\alpha}{\eta}+\alpha L)+\left|-\frac{2\alpha\beta}{\eta}-\frac{2\alpha}{\eta}\left(\gamma-\frac{\alpha\beta}{\eta}\right)\right|\right)\|z^{k-1}-z^*\|^2 \\
         && + \left(\alpha^2L^2+\frac{\alpha^2}{\eta^2}+\frac{\alpha\tau L}{\eta}-\frac{2\alpha}{\eta}+2\alpha\mu+\alpha\tau L^2+\left|-\frac{\alpha\beta}{\eta}-\frac{\alpha}{\eta}\left(\gamma-\frac{\alpha\beta}{\eta}\right)\right|\right)\|z^k-z^{k+0.5}\|^2 \\
         && + \left(-2\alpha+\frac{2\alpha^2}{\eta}\right)(z^k-z^{k+0.5})^{\top}\left(F(z^{k+0.5})-F(z^k)\right) \\
         && -2\alpha\left(\gamma-\frac{\alpha\beta}{\eta}\right)(F(z^{k+0.5})-F(z^k))^{\top}(z^k-z^{k-1}) \\
         && -2\tau\left(\gamma-\frac{\alpha\beta}{\eta}\right)\left(F(z^k)-F(z^{k-1})\right)^{\top}(z^k-z^{k-1}) .
\end{eqnarray*}

\subsection{Proof of Lemma \ref{lem:res-extra-per-iter-conv}}
\label{app:res-extra-per-iter-conv}
Let us first present an inequality that we need to use in the proof, which is derived from the optimality condition of the update $z^{k+0.5}$. We have
\begin{equation*}
    \langle z^{k+0.5}-z^k+\eta F(z^k)-\beta(z^k-z^{k-1}),z-z^{k+0.5}\rangle \geq 0,\quad\forall z\in\mathcal{Z}.
\end{equation*}
By rearranging the terms we get
\begin{eqnarray}
    \label{opt-z-0.5}
        & &  \langle -\eta F(z^k)+\beta(z^k-z^{k-1}),z-z^{k+0.5}\rangle \nonumber \\
        & \leq & (z^{k+0.5}-z^k)^{\top}(z-z^{k+0.5})  \nonumber \\
         & = & \frac{1}{2}\left(\|z-z^k\|^2-\|z^{k+0.5}-z^k\|^2-\|z-z^{k+0.5}\|^2\right),\quad\forall z \in\mathcal{Z}.
\end{eqnarray}

We are now ready to establish the convergence analysis of the extra-point scheme with restricted domain:
\begin{eqnarray*}
       & & \|z^{k+1}-z^*\|^2 \\ & = & \|P_{\mathcal{Z}}\left(z^k-\alpha F(z^{k+0.5})+\gamma(z^k-z^{k-1})-\tau\left(F(z^{k})-F(z^{k-1})\right)\right)-P_{\mathcal{Z}}(z^*)\|^2 \\
         & \leq & (z^{k+1}-z^*)^{\top}\left[z^k-\alpha F(z^{k+0.5})+\gamma(z^k-z^{k-1})-\tau\left(F(z^{k})-F(z^{k-1})\right)-z^*\right] \\
         & = & \frac{1}{2}\|z^{k+1}-z^*\|^2+\frac{1}{2}\|z^k-z^*\|^2-\frac{1}{2}\|z^{k+1}-z^k\|^2-\tau(z^{k+1}-z^*)^{\top}\left(F(z^{k})-F(z^{k-1})\right) \\
         && +[-\alpha F(z^{k+0.5})+\gamma(z^k-z^{k-1})]^{\top}(z^{k+1}-z^*).
\end{eqnarray*}
We shall bound the last term with the following:
\begin{eqnarray*}
        & &  [-\alpha F(z^{k+0.5})+\gamma(z^k-z^{k-1})]^{\top}(z^{k+1}-z^*) \\
        & = & [-\alpha F(z^{k+0.5})+\gamma(z^k-z^{k-1})]^{\top}(z^{k+1}-z^{k+0.5}+z^{k+0.5}-z^*)  \\
         & = & \underbrace{[-\eta F(z^{k})+\beta(z^k-z^{k-1})]^{\top}(z^{k+1}-z^{k+0.5})}_{(a)} \\
         && + \underbrace{[-\alpha F(z^{k+0.5})+\eta F(z^k)+(\gamma-\beta)(z^k-z^{k-1})]^{\top}(z^{k+1}-z^{k+0.5})}_{(b)} \\
         && +\underbrace{[-\alpha F(z^{k+0.5})+\gamma(z^k-z^{k-1})]^{\top}(z^{k+0.5}-z^*)}_{(c)},
\end{eqnarray*}
where
\begin{equation*}
    (a) \overset{\eqref{opt-z-0.5}}{\leq} \frac{1}{2}\|z^{k+1}-z^k\|^2-\frac{1}{2}\|z^{0.5}-z^k\|^2-\frac{1}{2}\|z^{k+1}-z^{k+0.5}\|^2,
\end{equation*}
and
\begin{eqnarray*}
         (b) & = & [-\alpha F(z^{k+0.5})+\alpha F(z^k)-\alpha F(z^k)+\eta F(z^k)+(\gamma-\beta)(z^k-z^{k-1})]^{\top}(z^{k+1}-z^{k+0.5})  \\
         & \leq & \frac{1}{2}\alpha L\|z^{k+0.5}-z^k\|^2+\frac{1}{2}\alpha L\|z^{k+1}-z^{k+0.5}\|^2 \\
         & & +(\eta-\alpha)F(z^k)^{\top}(z^{k+1}-z^{k+0.5})+(\gamma-\beta)(z^k-z^{k-1})^{\top}(z^{k+1}-z^{k+0.5}),
\end{eqnarray*}
and
\begin{eqnarray*}
         (c) & \leq & -\alpha\mu\|z^{k+0.5}-z^*\|^2+\gamma(z^k-z^{k-1})^{\top}(z^{k+0.5}-z^*)  \\
         & \leq & -\frac{1}{2}\alpha\mu\|z^{k}-z^*\|^2+\alpha\mu\|z^{k+0.5}-z^k\|^2+\gamma(z^k-z^{k-1})^{\top}(z^{k+0.5}-z^*).
\end{eqnarray*}

Using the bound for (a), (b), and (c), we have:
\begin{eqnarray*}
        \|z^{k+1}-z^*\|^2 & \leq& \|z^k-z^*\|^2-\|z^{k+1}-z^k\|^2-2\tau(z^{k+1}-z^*)^{\top}\left(F(z^{k})-F(z^{k-1})\right) + 2[(a)+(b)+(c)] \\
         & \leq & (1-\alpha\mu)\|z^k-z^*\|^2+(\alpha L-1)\|z^{k+1}-z^{k+0.5}\|^2\underbrace{-2\tau(z^{k+1}-z^*)^{\top}\left(F(z^{k})-F(z^{k-1})\right)}_{(i)}\\
         && +(\alpha L+2\alpha\mu-1)\|z^{k+0.5}-z^k\|^2+\underbrace{2\gamma(z^k-z^{k-1})^{\top}(z^{k+0.5}-z^*)}_{(ii)} \\
         && + 2(\eta-\alpha)F(z^k)^{\top}(z^{k+1}-z^{k+0.5})+\underbrace{2(\gamma-\beta)(z^k-z^{k-1})^{\top}(z^{k+1}-z^{k+0.5})}_{(iii)},
\end{eqnarray*}
where
\begin{eqnarray*}
    (i) &\leq& 2\tau L\|z^{k+1}-z^*\|\|z^{k}-z^{k-1}\| \\
    &\leq& \tau L\|z^{k+1}-z^*\|^2+\tau L\|z^{k}-z^{k-1}\|^2\leq \tau L\|z^{k+1}-z^*\|^2+2\tau L\left(\|z^{k}-z^{*}\|^2+\|z^{k-1}-z^*\|^2\right),
\end{eqnarray*}
and
\begin{eqnarray*}
         (ii) & \leq & \gamma\|z^k-z^{k-1}\|^2+\gamma\|z^{k+0.5}-z^*\|^2  \\
         & \leq & 2\gamma\left(\|z^k-z^*\|^2+\|z^{k-1}-z^*\|^2+\|z^{k+0.5}-z^k\|^2+\|z^k-z^*\|^2\right) \\
         & = & 2\gamma\left(2\|z^k-z^*\|^2+\|z^{k-1}-z^*\|^2+\|z^{k+0.5}-z^k\|^2\right),
\end{eqnarray*}
and
\begin{eqnarray*}
         (iii) & \leq & |\gamma-\beta|\left(\|z^k-z^{k-1}\|^2+\|z^{k+1}-z^{k+0.5}\|^2\right)  \\
         & \leq & |\gamma-\beta|\left(2\|z^k-z^*\|^2+2\|z^{k-1}-z^*\|^2+\|z^{k+1}-z^{k+0.5}\|^2\right).
\end{eqnarray*}
Combining the above three bounds, we finally arrive at \eqref{exp-1-ineq-result}.

\end{appendices}

\end{document}